\documentclass[a4paper]{amsart}
\usepackage{amsmath,amsthm, amscd, amssymb, amsfonts}
\usepackage[all]{xy}
\usepackage{leftidx}

\numberwithin{equation}{section}\theoremstyle{plain}

\newtheorem{theorem}{Theorem}[section]
\newtheorem{corollary}[theorem]{Corollary}

\newtheorem{proposition}[theorem]{Proposition}
\newtheorem{lemma}[theorem]{Lemma}

\theoremstyle{definition}
\newtheorem{definition}[theorem]{Definition}
\newtheorem{example}[theorem]{Example}

\theoremstyle{remark}
\newtheorem{remark}[theorem]{Remark}

\newcommand{\un}{\textbf{1}}

\newcommand{\C}{{\mathcal C}}
\newcommand{\D}{{\mathcal D}}

\newcommand{\Z}{{\mathcal Z}}

\newcommand{\TY}{\mathcal{TY}}

\newcommand\SL{\operatorname{SL}}
\newcommand{\G}{{\mathcal G}}
\newcommand{\E}{{\mathcal E}}

\newcommand{\cd}{\operatorname{cd}}

\newcommand{\Rep}{\operatorname{Rep}}

\newcommand{\KER}{\mathfrak{Ker}}

\newcommand\Aut{\operatorname{Aut}}
\newcommand\Irr{\operatorname{Irr}}
\newcommand\FPdim{\operatorname{FPdim}}

\newcommand\vect{\operatorname{Vec}}
\newcommand\svect{\operatorname{sVec}}
\newcommand\id{\operatorname{id}}

\newcommand\Infl{\operatorname{Infl}}

\newcommand\rev{\operatorname{rev}}

\newcommand\Hom{\operatorname{Hom}}

\newcommand\cs{\operatorname{cs}}

\newcommand\IP{\operatorname{IP}}

\begin{document}
\title[Frobenius-Perron graphs of an integral fusion category]{Graphs attached to simple Frobenius-Perron dimensions of an integral fusion category}
\author{Sonia Natale}
\author{Edwin Pacheco Rodr\' \i guez}
\address{Facultad de Matem\'atica, Astronom\'\i a y F\'\i sica.
Universidad Nacional de C\'ordoba. CIEM -- CONICET. Ciudad
Universitaria. (5000) C\'ordoba, Argentina}
\email{natale@famaf.unc.edu.ar, \emph{URL:}\/ http://www.famaf.unc.edu.ar/$\sim$natale}
\email{efpacheco@famaf.unc.edu.ar}
\thanks{The research of S. Natale is partially supported by  CONICET and SeCYT--UNC. The research of E. Pacheco Rodr\' \i guez is partially supported by  CONICET, SeCYT--UNC and ANPCyT--FONCYT}



\date{November 14, 2014}

\begin{abstract} Let $\C$ be an integral fusion category. We study some graphs, called the prime graph and the common divisor graph, related to the Frobenius-Perron dimensions of simple objects in the category $\C$, that extend the corresponding graphs associated to the irreducible character degrees and the conjugacy class sizes of a finite group. We describe these graphs in several cases, among others, when $\C$ is an equivariantization under the action of a finite group, a $2$-step nilpotent fusion category, and the representation category of a twisted quantum double.
We prove generalizations of known results on the number of connected components of the corresponding graphs for finite groups in the context of braided fusion categories. In particular, we show that if $\C$ is any  integral non-degenerate braided fusion category, then the prime graph of $\C$ has at most $3$ connected components, and it has at most $2$ connected components if $\C$ is in addition solvable.
As an application we prove a classification result for weakly integral braided fusion categories  all of whose simple objects have prime power Frobenius-Perron dimension.
\end{abstract}

\maketitle

\section{Introduction}

Throughout this paper we shall work over an algebraically closed base field $k$ of characteristic zero.

Let $G$ be a finite group. Several graphs that can be attached to the the set $\cd(G)$ of irreducible character degrees of $G$ over $k$ and to the set $\cs (G)$ of conjugacy class sizes in $G$, have been intensively studied. The knowledge of these graphs provides important information on the structure of the group $G$. See for instance \cite{camina}, \cite{lewis}, \cite{manz-wolf} and references therein.

The graphs we are going to consider in this paper are mainly the prime graph $\Delta(G)$  and the related common divisor graph $\Gamma(G)$. These  graphs are defined as follows: the vertex set of $\Delta(G)$ is the set of prime numbers $p$ such that $p$ divides some irreducible character degree of $G$. Two vertices $p$ and $q$ are joined by an edge if the product $pq$ divides an irreducible character degree. On the other hand, the vertex set of the common divisor graph $\Gamma(G)$ is the set of nontrivial irreducible character degrees, and two vertices $a$ and $b$ are joined by an edge if $a$ and $b$ are not relatively prime.

Similar graphs can be attached to the set $\cs(G)$. The graphs thus obtained are denoted by $\Delta'(G)$ and $\Gamma'(G)$.

It is known that, for any finite group $G$, the graph $\Delta(G)$ has at most three connected components. In the case where  $G$ is solvable,  $\Delta(G)$ has at most two connected components; furthermore, if $\Delta(G)$ is connected, then its diameter is at most $3$, while if it is not connected, then each connected component is a complete graph \cite{manz-sw}, \cite{manz} \cite{manz-ww}.

On the other hand, the graph $\Delta^{\prime}(G)$ has at most two connected components for any group $G$, and it is not connected if and only if $G$ is a quasi-Frobenius group with abelian complement and kernel: in this case each connected component is a complete graph \cite{BHM}, \cite{kazarin}.

\medbreak A generalization of the notion of finite group is given by that of a \emph{fusion category}. A fusion category $\mathcal{C}$ over $k$ is a $k$-linear semisimple rigid tensor category with finitely many simple objects, finite dimensional Hom spaces such that the unit object $\un$ is simple. Thus, for instance, the category $\Rep G$ of finite-dimensional $k$-representations of a finite group $G$ is a fusion category over $k$.

Let $\C$ be an integral fusion category over $k$.
In this paper we study the prime graph $\Delta(\C)$ and the common divisor graph $\Gamma(\C)$ of $\C$, that we call the Frobenius-Perron graphs of $\C$. These are defined as the graphs associated with the set $\cd (\C)$ of Frobenius-Perron dimensions of simple objects of $\C$.

The graphs $\Delta(\C)$ and $\Gamma(\C)$ coincide with $\Delta(G)$ and $\Gamma(G)$, respectively, when  $\C$ is the category $\Rep G$ of finite-dimensional representations of $G$ over $k$. On the other hand, the prime graph of the category $\C$ of finite-dimensional representations of a twisted quantum double $D^{\omega}(G)$, $\omega \in H^3(G, k^*)$, appears related to the prime graph $\Delta'(G)$. In fact, $\Delta'(G)$ is always a subgraph of $\Delta(\Rep D^{\omega}(G))$, and we give some examples where these graphs coincide.

The classes of group-theoretical and weakly group-theoretical fusion categories were introduced in the papers \cite{ENO}, \cite{ENO2}, relying on certain notions of group extensions of fusion categories. A fusion category $\C$ is called group-theoretical if it is Morita equivalent to a pointed fusion
category, and it is called  weakly group-theoretical if it is Morita equivalent to a nilpotent fusion category. If $\C$ is moreover Morita equivalent to a cyclically nilpotent fusion category, then it is called solvable. We shall recall these definitions in Section \ref{fc}.

An important open question related to the classification of fusion categories is whether any fusion category whose Frobenius-Perron dimension is a natural integer is weakly group-theoretical \cite[Question 2]{ENO2}.

\medbreak Let $G$ be a finite group and let $\C$ be an integral fusion category endowed with an action of $G$ by tensor autoequivalences. We prove several results on the prime graph of the equivariantization $\C^G$. Under the assumption that the group $G$ is not abelian, we show that the graph $\Delta(\C^G)$ has at most three connected components (Theorem \ref{thm-leq3}), and it has at most two connected components if $G$ is solvable (Theorem \ref{equiv-solv}). In order to prove Theorems \ref{thm-leq3} and \ref{equiv-solv} we adapt some arguments used in the proofs of the corresponding results for the character degree graphs of a finite group in \cite[Theorem 18.4]{manz-wolf} and \cite[Proposition 2]{manz-sw}, respectively. An important tool in our proof is an analogue of a theorem of Gallagher for fusion categories that we establish in Proposition \ref{gallagher} (see also Corollary \ref{gallagher-exact}). Moreover, we show that if $\C$ is pointed then $\Delta(\C^G)$ has at most three connected components, for any group $G$ (Theorem \ref{eq-pted}). This is done by applying properties of another graph, called the IP-graph, introduced by Isaacs and Praeger in \cite{ipgraph}

\medbreak Among all fusion categories, a distinguished class is that of braided fusion categories and in particular, that of  non-degenerate braided fusion categories. These notions are recalled in Section \ref{braided}. Every non-degenerate fusion category with integral Frobenius-Perron dimension admits a canonical spherical structure that makes it a \emph{modular} category. This type of categories are of relevance in distinct areas of mathematics and mathematical physics. See for instance \cite{BK}, \cite{turaev-b}.

\medbreak As a consequence of the above mentioned results for equivariantizations, we obtain the following theorems in the context of integral braided fusion categories, which are analogues of the corresponding results for finite groups:

\begin{theorem}\label{main-nondeg-wgt} Let $\C$ be an integral non-degenerate fusion category.
Then we have:
\begin{enumerate}\item[(i)] The graph $\Delta(\C)$ has at most three connected components.
\item[(ii)] Suppose $\C$ is solvable. Then the graph $\Delta(\C)$ has at most two connected components.
\end{enumerate}
\end{theorem}

Let us observe that integral non-degenerate categories which are solvable and such that $\Delta(\C)$ has two connected components do exist. See Examples \ref{ej-36} and \ref{ej-zty}.

\medbreak For general braided group-theoretical fusion categories, which are always integral, we show:

\begin{theorem}\label{main-braidedgt} Let $\C$ be a braided group-theoretical fusion category. Then we have:
\begin{enumerate}
\item[(i)] The graph $\Delta(\C)$ has at most three connected components.
\item[(ii)] If $\C$ is non-degenerate, then the graph  $\Delta(\C)$ has at most two connected components and its diameter is at most $3$. \qed
\end{enumerate} \end{theorem}

The proofs of Theorems \ref{main-nondeg-wgt} and \ref{main-braidedgt} are given in Subsection \ref{pfs}. Theorem \ref{main-braidedgt} follows from a description of the graphs of the twisted quantum double of a finite group (Theorem \ref{tqd}) and that of an equivariantization of a pointed fusion category (Theorem \ref{eq-pted}).  The proof of Theorem \ref{main-nondeg-wgt} relies on a study of Tannakian subcategories of braided fusion categories and its connection with the equivariantization construction. We make use in the proofs the fact that the Frobenius-Perron graph of an integral non-degenerate braided fusion category without non-pointed Tannakian subcategories has at most two connected components. See Proposition \ref{nd-connect}. This result follows from an application of the  Verlinde formula for modular categories.

\medbreak Let $G$ be a finite group and assume that all irreducible degrees of $G$ are prime powers.  A result of  Willems \cite{willems} says that $G$ must be solvable,  unless $G \cong S \times A$ where
$A$ is an abelian group and $S$ is one of the groups $\mathbb A_5$ or $\SL(2,8)$. An alternative proof is given in \cite{manz-sw} as a consequence of the fact that the prime graph of $G$ has at most three connected components.

The connected components of the graph $\Delta(G)$ are $\{2\}, \{3\}, \{5\}$, if $G =\mathbb A_5$, and  $\{2\}, \{3\}, \{7\}$, if $G = \SL(2, 8)$.

\medbreak Let $\C$ be a braided fusion category such that $\FPdim \C \in \mathbb Z$. It was shown in \cite[Theorem 7.2]{witt-wgt} that if there exists a prime $p$ such that the Frobenius-Perron dimension of every simple object of $\C$ is a power of $p$, then $\C$ is solvable. In the case where $\C$ is integral, this can be rephrased in terms of the prime graph of $\C$ saying that if $\Delta(\C)$ consists of a single isolated vertex, then $\C$ is solvable.

As an application of the main theorems of this paper, we prove in Section \ref{application} the following generalization of these results:

\begin{theorem}\label{appl} Let $\C$ be a braided fusion category such that $\FPdim \C \in \mathbb Z$ and let $p_1, \dots, p_r$ be prime numbers. Suppose that the Frobenius-Perron dimensions of any simple object of $\C$ is a $p_i$-power, for some $1 \leq i \leq r$. Then $\C$ is weakly group-theoretical.

Assume in addition that one of the following conditions is satisfied:
\begin{enumerate}
\item[(a)] $r \leq 2$, or
\item[(b)] $p_i > 7$, for all $i = 1, \dots, r$.
\end{enumerate} Then $\C$ is solvable.  \end{theorem}

\medbreak The paper is organized as follows. In Section \ref{preliminaries} we recall the definitions and some of the main properties of the graphs associated to irreducible character degrees and conjugacy class sizes of finite groups. Another graph associated to transitive actions of groups, called the IP-graph, is also discussed in this section.

In Section \ref{fc} we give an account of the relevant definitions and several features of fusion categories that will be needed in the rest of the paper. In particular, we discuss in this section the notions of group extensions and equivariantizations, and the notions of exact sequence of fusion categories and its relation with the equivariantization under a finite group action.
Notions and results about braided fusion categories are discussed later on at the beginning of Section \ref{braided}.

In Section \ref{fp-graph} we introduce the Frobenius-Perron graphs of an integral fusion category and
give some examples. The  Frobenius-Perron graph of some classes of graded extensions, namely, $2$-step nilpotent fusion categories and braided nilpotent fusion categories, are described in Section \ref{fp-extensions}.

In Section \ref{s-gallagher} we prove the mentioned analogue of Gallagher's Theorem. After that we prove in Section \ref{graph-equiv} our main results on the Frobenius-Perron graphs of equivariantizations of fusion categories. Our results concern equivariantizations under the action of a non-abelian group and equivariantizations of pointed fusion categories.
The graph of the representation category of a twisted quantum double and its relation with the graphs $\Delta(G)$ and $\Delta'(G)$ of a finite group $G$ are described in Section \ref{s-tqd}. Section \ref{braided} is devoted to the proof of the main results on the Frobenius-Perron graphs of braided fusion categories. The last section contains the applications of the main results to the classification of certain braided fusion categories.

\subsection*{Acknowledgement} The research of S. Natale was done in part during a stay at the Erwin Schr\" odinger Institute, Vienna, in the frame of the Programme 'Modern Trends in Topological Quantum Field Theory' in February 2014. She thanks the ESI and the organizers of the Programme for the support and kind hospitality.

\section{Some graphs associated to a finite group}\label{preliminaries}

Throughout this section $G$ will be a finite group. We shall recall here the definitions and main properties of some graphs associated to $G$.

\medbreak Let $n$ be a natural number. We shall use the notation $\pi(n)$ to indicate the set of prime divisors of $n$. If $X$ is a set of positive integers, the notation $\pi(X)$ will indicate the set of prime numbers $p$ such that $p$ divides an element of $X$.

\medbreak All graphs considered in this paper are non-oriented graphs. For a graph $\G$, the vertex set of $\G$ will be denoted by $\vartheta(\G)$. Let  $x, y$ be two vertices connected by a path in the graph $\G$. The \emph{distance} between $x$ and $y$ will be denoted by $d(x, y)$; by definition, $d(x, y)$ is the minimum length of a path in $\G$ connecting $x$ and $y$. The \emph{diameter} of $\G$ is the maximum among all distances between vertices in $\G$.

\begin{definition}  Let $S$ be a set of positive integers. The \emph{prime vertex graph} $\Delta(S)$ and the \emph{common divisor graph} $\Gamma (S)$ of $S$ are the graphs defined as follows:

\medbreak The graph $\Delta(S)$ has as vertex set the set $\pi(S)$. Two vertices $p$ and $q$ are joined by an edge if and only if exist $a \in S$ such that $pq$ divides $a$.

\medbreak The graph $\Gamma(S)$ has vertex set $S - \{1\}$. Two vertices $a$ and $b$ are joined by an edge if and only if $a$ and $b$ are not coprime. \end{definition}

\begin{remark}\label{graphs-sets} We have that the number of connected components of $\Gamma(S)$ is equal to the number of connected components of $\Delta(S)$; see for instance \cite[Corollary 3.2]{lewis}.
\end{remark}

\subsection{The graph $\Delta(G)$}\label{prime-G}  Let $\cd(G)$ be the set of irreducible character degrees of the group $G$, that is,  $$\cd(G) = \{\chi(1) : \chi \in \Irr(G)\}.$$ The set $\pi(\cd(G))$ is thus the set of prime numbers $p$ such that $p$ divides $\chi(1)$ for some $\chi \in \Irr(G)$.
We shall use the notation $\Delta(G)$ to indicate the prime vertex graph of the set $\cd(G)$.
The relations between the structure of the graph $\Delta(G)$ and the structure of $G$ have been studied extensively. See for instance \cite[Chapter V]{manz-wolf}, \cite{lewis} and references therein.

In the following theorem, we summarize some results due to Manz, Staszewski and Willems \cite{manz-sw}, Manz \cite{manz} and Manz, Willems and Wolf \cite{manz-ww}; see \cite[Corollary 4.2 and Theorem 6.4]{lewis}:

\begin{theorem}\label{ppties-graphg-1} Let $G$ be a finite group. Then the following hold:
 \begin{enumerate}
  \item[(i)] The graph $\Delta(G)$ has at most three connected components.
       \end{enumerate}
Suppose that the group $G$ is solvable. Then we have:
 \begin{enumerate}
  \item[(ii)] The graph $\Delta(G)$ has at most two connected components.
  \item[(iii)] If $\Delta(G)$ is connected, then its diameter is at most $3$. \qed
       \end{enumerate}
\end{theorem}

We point out that the proof of part (i) of the theorem given in \cite{manz-sw} relies on the classification of finite simple groups.

When $G$ is solvable, a result of P\' alfy asserts moreover that if the graph $\Delta(G)$ is not connected, then each connected component is a complete graph.

\subsection{The graph $\Delta^{\prime}(G)$}  Let $\cs(G)$ be the set $$\cs(G) = \{|C| : C \textrm{ is a conjugacy class of  } G\}.$$

The prime graph and the common divisor graph of the set $\cs(G)$ will be  denoted, respectively, by $\Delta^{\prime}(\emph{G})$ and $\Gamma'(G)$.
In this paper we shall use some important properties of the graph $\Delta'(G)$, that we list in Theorem \ref{ppties-graphg} below. These results can be found in the surveys \cite{camina} and \cite{lewis}.



\begin{remark} The vertex set of the graph $\Delta(G)$ is contained in the vertex set of the graph $\Delta'(G)$, in other words, $\pi(\cd(G)) \subseteq \pi(\cs(G))$.

\medbreak Indeed, if $p$ is a prime divisor of the order of $G$, the condition $p \notin \pi(\cs G)$ is equivalent to the existence of a central Sylow $p$-subgroup in $G$. (Suppose $p$ does not divide  $|^Ga|$ for all $a \in G$, then $G$ has a central Sylow $p$-subgroup $S_p$ by \cite[Corollary 4]{camina}. The other implication follows from the fact that if $G$ has a central Sylow $p$-subgroup $S_p$, then $S_p$ is an abelian direct factor of $G$ and therefore  $\cs(G) = \cs(G/S_p)$. But the numbers in $\cs(G/S_p)$ all divide the order of $G/S_p$, which is relatively prime to $p$.)

Assume $p \notin \pi(\cs G)$. Since $G$ has a central Sylow $p$-subgroup then, by a result of Ito \cite{ito}, the degree of any irreducible character divides the index $[G: S_p]$ and therefore it is not divisible by  $p$.  Hence $\pi(\cd(G)) \subseteq \pi(\cs(G))$, as claimed.

\medbreak Let us observe that the condition $p \notin \pi(\cd G)$ is equivalent to the existence of a normal abelian Sylow $p$-subgroup, in view of Ito's theorem and a result of Michler \cite[Theorem 5.4]{michler}.

\medbreak A result of Dolfi establishes that, when $G$ is a solvable group, $\Delta(G)$ is in fact a subgraph of $\Delta'(G)$. See \cite[Theorem 8.4]{lewis}.
\end{remark}

\medbreak Recall that a group $G$ is called a  \emph{Frobenius} group with \emph{Frobenius complement} $H$ if there exists a proper subgroup $H \subseteq G$ such that $H \cap {}^g\!H = \{ e\}$ for all $g \in G\backslash H$. If $G$ is a Frobenius group with Frobenius complement $H$, then $G$ is a semi-direct product $G = N\rtimes H$, where the normal subgroup $N \subseteq G$, which is uniquely determined by $H$, is called the \emph{Frobenius kernel} of $G$. If $G$ is a Frobenius group with complement $H$ and kernel $N$, then the orders of $H$ and $N$ are relatively prime.

More generally, a group $G$ is called a \emph{quasi-Frobenius} group with Frobenius kernel $\tilde N \subseteq G$ and Frobenius complement $\tilde H \subseteq G$, if  $G/Z(G)$ is a Frobenius group with kernel $\tilde N/Z(G)$ and complement $\tilde H/Z(G)$.

\medbreak The results listed in the following theorem are due to  Bertram, Herzog and Mann, Kazarin, Casolo and Dolfi and Alfandary \cite{BHM}, \cite{kazarin}, \cite{casolo-dolfi}, \cite{alfandary}. See \cite[Section 8]{lewis}.

\begin{theorem}\label{ppties-graphg} Let $G$ be a finite group. Then the following hold:
 \begin{enumerate}
  \item[(i)] The graph $\Delta^{\prime}(G)$ has at most two connected components.
  \item[(ii)] The graph $\Delta^{\prime}(G)$ has two connected components if and only if $G$ is a quasi-Frobenius group with abelian complement and kernel. In this case each connected component is a complete graph.
  \item[(iii)] If the group $G$ is not solvable, then the graph $\Delta^{\prime}(G)$ is connected and its diameter is at most two.
  \item[(iv)] If the graph $\Delta^{\prime}(G)$ is connected, then its diameter is at most three.
\qed       \end{enumerate}
\end{theorem}

\begin{remark} Let $G$ be a quasi-Frobenius group with abelian kernel $\tilde N$ and abelian complement $\tilde H$.
The conjugacy class sizes of $G$ are the same as those of the Frobenius group $G/Z(G)$, that is, $\cs(G) = \cs(G/Z(G))$. Therefore $\Delta'(G) = \Delta'(G/Z(G))$.

Let $n$ and $m$ be the orders of the groups $N = \tilde N/Z(G)$ and $H =\tilde H/Z(G)$, respectively, so that $(n, m) = 1$. If $a \in G$, then we have $|^Ga| = n$ or $m$; see \cite[Section 2]{camina}. Therefore in this case the two connected components of graph $\Delta^{\prime}(G)$ are the complete graphs on the vertex sets $\pi(n)$ and $\pi(m)$, respectively.
\end{remark}

\subsection{The graph $\Gamma'(G)$ and the $\IP$-graph}

We shall consider in this subsection a generalization of the graph $\Gamma'(G)$ introduced by Isaacs and Praeger in \cite{ipgraph} and known as the $\IP$-graph.

Let $G$ be a group acting transitively on a set $\Omega$. Consider an element $\alpha\in\Omega$ and let $G_{\alpha}$ be its stabilizer subgroup. The subgroup $G_{\alpha}$ acts on the set $\Omega$ by restriction of the original action.

\medbreak The subdegrees of $(G,\Omega)$ are the cardinalities of the orbits of the action of a stabilizer $G_{\alpha}$, $\alpha \in \Omega$, on the set $\Omega$.

The set of subdegrees of $(G,\Omega)$ is denoted by $D =D(G,\Omega)$. Because of the transitivity of the original action, the set $D$ is well-defined, independently of the choice of $\alpha\in\Omega$.

\begin{definition}[\cite{ipgraph}] Suppose that all subdegrees of $(G,\Omega)$ are finite. Then the \emph{$\IP$-graph} of $(G,\Omega)$ is the common divisor graph $\Gamma(D)$ of $D$.
\end{definition}

The main properties of the $\IP$-graph are summarized in the following theorem. See \cite[Theorem A and Theorem C]{ipgraph}.

\begin{theorem}\label{iptheorem} Let $G$ be a group. Suppose that $G$ acts transitively on the set $\Omega$ and that all subdegrees are finite. Then the $\IP$-graph of $(G,\Omega)$ has at most two connected components. Moreover, the following hold:
    \begin{enumerate}
     \item If the $\IP$-graph has just one connected component, then this component has diameter at most 4.
     \item If the $\IP$-graph has two connected components, one of these is a complete graph and the other has diameter at most 2.
    \end{enumerate}
\end{theorem}

\begin{example}\label{orbits} Theorem \ref{iptheorem} will be applied later on in this paper (see Theorem \ref{eq-pted}) in the following context. Let $G, A$ be finite groups such that $G$ acts on $A$ by automorphisms.
As explained in \cite{ipgraph}, a consequence of Theorem \ref{iptheorem} is that common divisor graph  of the set  of orbit sizes of the action of $G$ on $A$ has at most two connected components. See \cite[Corollary B]{ipgraph}.

\medbreak If, furthermore, the orders of $G$ and $A$ are relatively prime, then a result of T. Yuster (\cite{yuster}) implies that the common divisor graph  of the set  of orbit sizes is connected and its diameter is at most $2$.
\end{example}

\section{Fusion categories}\label{fc}

In this section we discuss several definitions and results on fusion categories that will be needed throughout the paper. We refer the reader to \cite{ENO}, \cite{ENO2}.

Let $\C$ be a fusion category over $k$. A \emph{fusion subcategory} $\D\subset \mathcal{C}$ is a full tensor subcategory with the following property: if $X\in \mathcal{C}$ is isomorphic to direct summand of an object of $\D$ then $X\in\mathcal D$. It is shown in \cite[Appendix F]{DGNOI} that such a subcategory is rigid, so a fusion subcategory is itself a fusion category.

Let $\C, \D$ be fusion categories over $k$. A \emph{tensor functor} $F:\C \to \D$ is a strong monoidal $k$-linear exact
functor.

\medbreak Let $\mathcal{C}$ be a fusion category, and let $K(\mathcal{C})$ be its Grothendieck ring. We shall denote by $\Irr(\mathcal{C})$ the set of isomorphism classes of simple objects, so that $\Irr(\C)$ is a basis of $K(\C)$. For every $X, Y, Z \in \Irr(\C)$, let also $N_{XY}^Z \in \mathbb Z_+$ denote the multiplicity of the simple object $Z$ in the tensor product $X \otimes Y$, so that $N_{XY}^Z = \dim \Hom_\C(Z, X \otimes Y)$, and we have a decomposition $$X \otimes Y \cong \bigoplus_{Z \in \Irr(\C)} N_{XY}^Z \, Z.$$

There exists a unique ring homomorphism $\FPdim\!: K(\mathcal{C})\rightarrow \mathbb{R}$,  such that $\FPdim(X) > 0$, for all $X\in\Irr (\C)$. The number $\FPdim(X)$ is called the \emph{Frobenius-Perron dimension} of the object $X$, it is the largest positive eigenvalue of the matrix $N^X$ of left multiplication by $X$ in $K(\mathcal{C})$. See \cite[Section 8.1]{ENO}.

The Frobenius-Perron dimension of $\C$ is defined as $$\FPdim \C =
\sum_{X \in \Irr(\C)} (\FPdim X)^2.$$ The category $\C$ is called \emph{integral}  if
$\FPdim X \in \mathbb Z$, for all simple object $X \in \C$, and it is called
\emph{weakly integral} if $\FPdim \C \in \mathbb Z$.

\medbreak An object $g$ of $\C$ is \emph{invertible} if  $g \otimes g^* \cong \un$ (equivalently, if the Frobenius-Perron dimension of $g$ is 1). The set of isomorphism classes of invertible objects of $\C$ is a subgroup of the group of units of $K(\C)$, and we have $g^{-1} = g^*$, for all invertible object $g$. Let $X, Y \in \Irr(\C)$ and let $g$ be an invertible object. The multiplicity of  $g$ in the tensor product $X \otimes Y$ is either 0 or 1, and $g$ is a constituent in $X \otimes Y^*$ if and only if $g \otimes Y \cong X$.

\medbreak The category $\C$ is called \emph{pointed} if all its simple objects are invertible.   Every pointed fusion category is equivalent to the category $\C(G, \omega)$ of finite-dimensional vector spaces graded by a finite group $G$, with associativity constraint given by a $3$-cocycle $\omega:G \times G \times G \to k^*$.

The full subcategory of $\C$ generated by its invertible objects is the largest pointed fusion subcategory of $\C$ and it is denoted by $\C_{pt}$.

\subsection{Exact sequences and equivariantization}\label{exact-seq}
Let $\C, \D$ be fusion categories and let $F:\C \to \D$ be a tensor functor. The functor $F$ is \emph{dominant} if any object $Y$ of $\D$ is a subobject of $F(X)$ for some object $X\in \C$.

Let $\KER_F \subset \C$ denote the fusion subcategory of objects $X$ of $\C$ such that $F(X)$ is a trivial object of $\D$, that is, such that $F(X)$ is isomorphic to $\un^{(n)}$ for some natural integer $n$.
The functor $F$ is \emph{normal} if for every simple object $X \in \C$ such that $\Hom_\D(\un, F(X)) \neq 0$, we have that $X \in \KER_F$.

\medbreak Recall from \cite{tensor-exact} that an \emph{exact sequence of fusion categories} is a sequence of tensor functors between fusion categories
\begin{equation}\label{suite}\C' \overset{i}\to \C \overset{F}\to \C''
\end{equation}
if the functor $F$ is dominant and normal, and $i$ is a full embedding whose essential image is  $\KER_F$.

\medbreak Let  $G$ be a finite group. Denote by $\underline{G}$ the monoidal category whose objects are the elements of $G$,  morphisms are identities and the tensor product is given by the multiplication in $G$.

Let $\C$ be a fusion category and let $\underline{\Aut}_{\otimes}(\mathcal{C})$ be the monoidal category of tensor autoequivalences $\mathcal{C}$.  An \emph{action by tensor autoequivalences} of $G$ on  $\mathcal{C}$ is a monoidal functor $\rho: \underline{G}\rightarrow \underline{\Aut}_{\otimes}(\mathcal{C})$. In other words, for every $g\in G$, there is a $k$-linear functor $\rho^g:\mathcal{C}\rightarrow\mathcal{C}$ and natural isomorphisms of tensor functors
    $$\rho_2^{g,h}:\rho^g\rho^h\rightarrow\rho^{gh}, \; g, h\in G,\quad \rho_0: id_{\mathcal{C}}\rightarrow\rho^e, \hspace{2.9cm} $$
satisfying the following conditions:
$$(\rho_2^{gh,l})_X(\rho_2^{g,h})_{\rho^l(X)} = (\rho_2^{g,hl})_X\rho^g((\rho_2^{h,l})_X),$$
$$(\rho_2^{g,e})_X\rho_{\rho_0(X)}^g = (\rho_2^{e,g})_X(\rho_0)_{\rho^g(X)}$$
for every object $X$ of $\mathcal{C}$ and for all $g, h, l\in G$.

Let $\rho: \underline{G}\rightarrow \underline{\Aut}_{\otimes}(\mathcal{C})$ be an action of $G$ on $\mathcal{C}$ by tensor autoequivalences. A \emph{G-equivariant object} of $\mathcal{C}$ is a pair $(X,\{\mu^g\}_{g\in G})$ where $X$ is an object of $\mathcal{C}$, and   $\mu^g:\rho^g(X)\simeq X$, $g\in G$, is a collection of isomorphisms such that
$$\mu^g\rho^g(\mu^h)=\mu^{gh}(\rho_2^{g,h})_X, \hspace{1.5cm} \mu_e\rho_{0X} =\id_X,$$ for all $g,h\in G$, $X \in \C$.

The \emph{equivariantization} $\mathcal{C}^G$ of $\C$ under the action of $G$ is the category whose objects are $G$-equivariant objects of $\mathcal{C}$, and  morphisms $f:(X,\mu)\rightarrow(X',\mu^{\prime})$ are morphisms $f:X\rightarrow X^{\prime}$ in $\mathcal{C}$ such that $f\mu^g = \mu^{\prime g}\rho^g(f)$ for all $g\in G$.
The equivariantization of a fusion category is again a fusion category.

Is shown in \cite[Section 5.3]{tensor-exact} that the forgetful functor $F: \C^G \to \C$, $F(X, \mu) = X$,
is a normal dominant tensor functor that gives rise to an exact sequence of
fusion categories
\begin{equation}\label{ees}\Rep G \to \C^G \to \C. \end{equation}

\subsection{Weakly group-theoretical and solvable fusion categories}

Let $\mathcal{C}$ be a fusion category and let $G$ be a finite group. A $G$-\emph{grading} on $\mathcal{C}$ is a decomposition into a direct sum of full abelian subcategories $\mathcal{C}=\bigoplus_{g\in G}\mathcal{C}_g$, such that $\mathcal{C}^*_g=\mathcal{C}_{g^{-1}}$ and, for all $g,h\in G$, the tensor product maps $\mathcal{C}_g \times \mathcal{C}_h$ to $\mathcal{C}_{gh}$.

The \emph{neutral component}  $\mathcal{C}_e$ of the grading is a fusion subcategory. If $\mathcal{C}_g\neq 0$, for all $g\in G$, then grading is called \emph{faithful}; in this case, $\mathcal{C}$ is called an $G$-\emph{extension} of $\mathcal{C}_e$.

\medbreak The notion of nilpotency of a fusion category was introduced in the paper \cite{gel-nik}. A fusion category $\mathcal{C}$ is \emph{nilpotent} if there exists a sequence of fusion categories $\mathcal{C}_0= \vect, \mathcal{C}_1, \ldots, \mathcal{C}_n = \mathcal{C}$ and a sequence $G_1,\ldots, G_n$ of finite groups such that $\mathcal{C}_i$ is obtained from $\mathcal{C}_{i-1}$ by a $G_i$-extension. $\mathcal{C}$ is called \emph{cyclically nilpotent} if the groups $G_i$ can be chosen to be cyclic.

\medbreak Let $\mathcal{C}$ be a fusion category. A \emph{left module category} over $\mathcal{C}$ (or $\C$-\emph{module category}) is a category $\mathcal{M}$ equipped with an action bifunctor $\otimes:\mathcal{C}\times\mathcal{M}\rightarrow\mathcal{M}$ and natural isomorphisms
$$ m_{X,Y,M}:(X\otimes Y)\otimes M\rightarrow X\otimes(Y\otimes M),\hspace{1.5cm} u_M:\textbf{1}\otimes M\rightarrow M, $$
$X, Y \in \C$, $M \in \mathcal M$, satisfying appropriate coherence conditions.  A $\mathcal{C}$-module category is called \emph{indecomposable} if it is not equivalent to a direct sum of two nontrivial $\mathcal{C}$-submodule categories.

Let $\mathcal{M}$ be an indecomposable $\mathcal{C}$-module category. Then the category $\mathcal{C}^*_{\mathcal{M}}$ of $\mathcal{C}$-module endofunctors of $\mathcal{M}$ is a fusion category.
A fusion category $\mathcal{D}$ is called \emph{(categorically) Morita equivalent} to $\mathcal{C}$  if there an indecomposable $\mathcal{C}$-module category $\mathcal{M}$ such that $\mathcal{D}\cong (\mathcal{C}^*_{\mathcal{M}})^{op}$.

\medbreak A fusion category $\C$ is called group-theoretical if it is categorically Morita equivalent to a pointed fusion category \cite{ENO}. Every group-theoretical fusion category is integral.

Weakly group-theoretical and solvable fusion categories were introduced in the paper \cite{ENO2}. A fusion category $\mathcal{C}$ is \emph{weakly group-theoretical} if it is Morita equivalent to a nilpotent fusion category.   On the other hand, $\mathcal{C}$ is \emph{solvable} if any of the following two equivalent conditions are satisfied:
\begin{itemize}
  \item[(i)] $\mathcal{C}$ is Morita equivalent to a cyclically nilpotent fusion category.
  \item[(ii)] There is a sequence of fusion categories $\mathcal{C}_0 = Vec, \mathcal{C}_1,\ldots, \mathcal{C}_n=\mathcal{C}$ and a sequence $G_1,\ldots, G_n$ of cyclic groups of prime order such that $\mathcal{C}_i$ is obtained from $\mathcal{C}_{i-1}$ by a $G_i$-equivariantization or as $G_i$-extension.
\end{itemize}

A weakly group-theoretical or solvable fusion category $\C$ is weakly integral, that is, $\FPdim \C \in \mathbb Z$, but not always integral. It is proved in \cite[section 4]{ENO2} that the class of weakly group-theoretical categories is closed under taking extensions, equivariantizations,  Morita equivalent categories, tensor products, subcategories, the center and component categories of quotient categories. Similarly, the class of solvable categories is closed under taking  Morita equivalent categories, tensor products, subcategories, the center, component categories of quotient categories and extensions and equivariantizations by solvable groups.

\section{The Frobenius-Perron graphs of an integral fusion category}\label{fp-graph}

Let $\C$ be an integral fusion category.  By analogy with the standard notation in the character theory of finite groups, we shall denote by $\cd(\C)$ the set
$$\cd(\C) = \{ \FPdim X:\, X \in \Irr(\C) \}.$$

\begin{definition} Let $\C$ be an integral fusion category. The prime graph $\Delta(\C)$ and the common divisor graph $\Gamma(\C)$ of $\C$ are, respectively, the prime graph and the common divisor graph on the set $\cd(\C) -\{ 1 \}$.

\medbreak We shall call $\Delta(\C)$ and $\Gamma(\C)$ the \emph{Frobenius-Perron graphs} of $\C$.
\end{definition}

It is clear that if $\D$ is a fusion subcategory of $\C$, then the graphs $\Delta(\D)$ and $\Gamma(\D)$ are subgraphs of $\Delta(\C)$ and $\Gamma(\C)$, respectively.

\medbreak
Let $G$ be a finite group. When $\C =\Rep G$ is the category of finite dimensional representations of $G$, the Frobenius-Perron graphs $\Delta(\C)$ and $\Gamma(\C)$ of $\C$ coincide with the graphs $\Delta (G)$ and $\Gamma(G)$ discussed in Subsection \ref{prime-G}.

\medbreak In the rest of this section we discuss the Frobenius-Perron graphs of some examples of integral fusion categories in the literature. Recall that a fusion category $\C$ is called \emph{of type} $(d_0,n_0;
d_1,n_1;\cdots;d_s,n_s)$, where $1=d_0, d_1,\cdots, d_s$, $s \geq 0$, are positive
real numbers such that $1 = d_0 < d_1< \cdots < d_s$, and 
$n_1,n_2,\cdots,n_s$ are positive integers,  if for all $i = 0, \cdots, s$, $n_i$ is
the number of the non-isomorphic simple objects of Frobenius-Perron dimension $d_i$ in $\C$.

\begin{example}\label{ej-36} Let $\C$ be one of the non-group theoretical modular categories $\C(\mathfrak{sl}_3, q, 6)$, where $q^2$ is a primitive $6$th root of unity \cite[Example 4.14]{NR}, \cite[Section 4.1]{int-mod}.  We have $\FPdim \C = 36$ and thus $\C$ is solvable in view of \cite[Theorem 1.6]{ENO2}.  The fusion category $\C$ is of type $(1, 3; 2, 6; 3, 1)$. Thus we see that the graph $\Delta(\C)$ consists of two isolated vertices $\Delta(\C) = \{ 2\} \cup \{ 3\}$.
\end{example}

\begin{example} Suppose $\C$ has a unique non-invertible simple object $X$; that is, $\C$ is a so-called \emph{near-group} fusion category. Assume in addition that $\FPdim X \in \mathbb Z$. Then $\Delta(\C)$ is the complete graph on the vertex set $\pi(\FPdim X)$ of prime divisors of $\FPdim X$. \end{example}

\begin{example}\label{ej-zty} Let $\C = \Z(\TY(A, \chi, \tau))$ be the Drinfeld  center of a Tambara-Yamagami fusion category $\TY(A, \chi, \tau)$, where $A$ is a finite abelian group, $\chi$ is a non-degenerate symmetric bicharacter on $A$ and $\tau =\pm 1$. The fusion category $\C$ is integral if and only if the order of $A$ is a square. The Frobenius-Perron dimension of $\C$ is $4n^2$, where $n = |A|$. In addition, being the center of a cyclically nilpotent fusion category, $\C$ is solvable.

It follows from \cite[Proposition 4.1]{GNN} that $\C$ is of type $(1, 2n; 2, \frac{n(n-1)}{2}; \sqrt{n}, 2n)$.
Assume that $\C$ is integral. Then the vertex set of the graph $\Delta(\C)$ consists of the prime $2$ and the prime divisors of $|A|$. Moreover, if $|A|$ is even, then $\Delta(\C)$ is a complete graph, while if $|A|$ is odd,   $\Delta(\C)$ has two connected components: one of them is the complete graph on the set of prime divisors of $|A|$ and the other consists of an isolated vertex $\{2\}$.

We point out that the same holds for the fusion subcategories $\E = \E(q, \pm)$ of $\C$ constructed in \cite[Section 5]{GNN}.
\end{example}

\section{Frobenius-Perron graph of some classes of graded extensions}\label{fp-extensions}

In this section we study the prime graph of some classes of nilpotent fusion categories.

\begin{proposition}\label{ext-pointed} Let $\C$ be an integral fusion category. Suppose that $\C$ is a $2$-step nilpotent fusion category, that is, $\C$ is an extension of a pointed fusion category. Then the graph $\Delta(\C)$ is connected and its diameter is at most $2$.
\end{proposition}

\begin{proof} By assumption there exists a finite group $G$ and a faithful grading $\C = \oplus_{g \in G}\C_g$ such that $\C_e \cong \D$ is a pointed fusion subcategory.

Suppose $X$ is a non-invertible simple object of $\C$. Then $X \in \C_g$, for some $e\neq g \in G$, and therefore $X \otimes X^* \in \C_e = \D$. Then all simple constituents of $X \otimes X^*$ are invertible, and it follows that $X \otimes X^* \cong \bigoplus_{s \in G[X]}s$, where $G[X]$ is the subgroup of the group of invertible objects $s$ of $\C$  such that $s \otimes X \cong X$. Hence the order of $G[X]$ equals $(\FPdim X)^2$ and thus it has the same prime divisors as $\FPdim X$.

Suppose $p \neq q$ are vertices of $\Delta(\C)$, so that there exist simple objects $X$ and $Y$ such that $p | \FPdim X$ and $q | \FPdim Y$. Then $p | |G[X]|$ and $q | |G[Y]|$. Suppose first that $|G[X]|$ and $|G[Y]|$ have a common prime divisor $r$. We may assume that $p \neq r \neq q$, and since $pr | \FPdim X$, $qr | \FPdim Y$, we have that $d(p, q) \leq 2$.

If, on the other hand, $|G[X]|$ and $|G[Y]|$ are relatively prime, then $G[X] \cap G[Y] = \un$ and therefore $X \otimes X^*$ and $Y \otimes Y^*$ have no nontrivial simple constituent. It follows from \cite[Lemma 2.5]{fusion-lowdim} that the tensor product $X^* \otimes Y$ is simple. Since $pq | \FPdim (X^* \otimes Y)$, then $d(p, q) =1$ in this case. This finishes the proof of the proposition. \end{proof}

\begin{example} As a special instance of the situation in Proposition \ref{ext-pointed}, consider the case where $\C$ is a fusion category with \emph{generalized Tambara-Yamagami fusion rules}: that is, $\C$ is a $\mathbb Z_2$-extension of a pointed fusion category.  By \cite[Lemma 5.1 (ii)]{faithful-fusion}, all non-invertible simple objects $X$ of $\C$ have the same stabilizer $G[X] = S$, where $S$ is a normal subgroup of the group of invertible objects of $\C$. In particular, $\C$ is integral if and only if the order of $S$ is a square, and in this case $\Delta (\C)$ is the complete graph in the vertex set of prime divisors of the order of $S$.
\end{example}

In the context of braided fusion categories we have the following:

\begin{proposition}\label{nilp} Let $\C$ be a braided nilpotent integral fusion category. Then $\Delta(\C)$ is a complete graph. \end{proposition}

\begin{proof} By \cite[Theorem 1.1]{DGNO}, there is an equivalence of braided fusion categories $\C \cong \C_1 \boxtimes \dots \boxtimes \C_m$, where for every $i =1, \dots, m$, $\C_i$ is a braided fusion category of prime power Frobenius-Perron dimension $p_i^{d_i}$, $d_i \geq 1$.

It is known that the Frobenius-Perron dimensions of simple objects in a weakly integral braided fusion category $\C$ divide the Frobenius-Perron dimension of $\C$ \cite[Theorem 2.11]{ENO2}.

Suppose $\C$ is integral. Then so are the fusion subcategories $\C_1, \dots, \C_m$.
Assume in addition that $\C$ is not pointed and $\C_1, \dots, \C_t$ are the non-pointed constituents, where $1 \leq t \leq m$. Then, for every $1\leq i \leq t$, $\C_i$ has a simple object $Y_i$ of Frobenius-Perron dimension divisible by $p_i$.

Note that every simple object $X$ of $\C$ is of the form  $X = X_1\boxtimes X_2 \boxtimes \dots \boxtimes X_m$, where $X_i$ is a simple object of $\C_i$, for all $i =1, \dots, m$. Hence the vertex set of $\Delta(\C)$ is the set $\{p_1, \dots, p_t\}$.

In addition, $Y = Y_1\boxtimes Y_2 \boxtimes \dots \boxtimes Y_t \boxtimes \un \boxtimes \dots \boxtimes \un$
is a simple object of $\C$, whose Frobenius-Perron dimension is divisible by $p_i$, for all $i =1, \dots, t$.
Hence we find that $\Delta(\C)$ is the complete graph on the vertex set $\{p_1, \dots, p_t\}$.
\end{proof}

\section{Gallagher's Theorem for fusion categories}\label{s-gallagher}

Let $\C, \D$ be fusion categories. Recall that the \emph{kernel} of a tensor functor $F:\C \to \D$ is the full subcategory $\KER_F = F^{-1}(\langle \un\rangle)$ of $\C$ whose objects are those $X \in \C$ such that $F(X)$ is a trivial object of $\D$ \cite[Section 3.1]{tensor-exact}.

\begin{proposition}\label{gallagher} Let $F:\C \to \D$ be a tensor functor between fusion categories $\C, \D$. Let $X, Y$ be simple objects of $\C$ such that $Y \in \KER_F$ and $F(X)$ is a simple object of $\D$. Then $Y \otimes X$ is a simple object of $\C$.
\end{proposition}

\begin{proof} By \cite[Lemma 2.5]{fusion-lowdim}, a necessary and sufficient condition for the tensor product $Y \otimes X$ to be simple is that for any simple object $Z \neq \un$ of $\C$, either $\Hom_\C(Z, Y^*\otimes Y) =0$ or $\Hom_\C(Z, X\otimes X^*) =0$.

Consider a simple object $Z \neq \un$ such that $\Hom_\C(Z, Y^*\otimes Y) \neq 0$.  Since $Y\in \KER_F$, this implies that also $Z\in \KER_F$, that is, $F(Z)$ is a trivial object of $\D$.

Suppose that $\Hom_\C(Z, X\otimes X^*)$ is also nonzero. Then we have a decomposition $X\otimes X^* \cong \un \oplus Z \oplus \dots$ in $\C$. Applying the functor $F$, we obtain isomorphisms  $F(X) \otimes F(X)^* \cong F(X\otimes X^*) \cong \un \oplus F(Z) \oplus \dots$ in $\D$. Since $F(Z) \neq 0$,  we get that the multiplicity of the trivial object of $\D$ in $F(X) \otimes F(X)^*$ is bigger than $1$. This contradicts the assumption that $F(X)$ is simple. Therefore $\Hom_\C(Z, X\otimes X^*) = 0$ for all such simple object $Z$, and $Y\otimes X$ is simple, as claimed. \end{proof}

Recall that a sequence of tensor functors $\C' \overset{f}\to \C \overset{F}\to \C''$ between tensor categories is called \emph{exact} if $F$ is a dominant normal tensor functor and $f$ is a full embedding whose essential image is $\KER_F$. See Subsection \ref{exact-seq}. We shall identify $\C'$ with $\KER_F$.
As an immediate consequence of Proposition \ref{gallagher}, we obtain:

\begin{corollary}\label{gallagher-exact} Let $\C' \to \C \overset{F}\to \C''$ be an exact sequence of fusion categories. Let also $Y \in \C'$, $X \in \C$, be simple objects and assume that $F(X)$ is a simple object of $\C''$. Then $Y \otimes X$ is a simple object of $\C$. \qed
\end{corollary}

\begin{remark} Consider an action $\underline{G} \to \underline{\Aut}_{\otimes}\C$ of a finite group $G$ on a fusion category $\C$ by  tensor autoequivalences, and let $\C^G$ denote the corresponding equivariantization. The forgetful functor $\C^G \to \C$ gives rise to an exact sequence of fusion categories $\Rep G \to \C^G \to \C$ \cite[Section 5.3]{tensor-exact}.

Corollary \ref{gallagher-exact} implies that for every finite-dimensional irreducible representation $V$ of $G$ and for every simple equivariant object $(X, u) \in \C^G$ whose underlying object $X$ is a simple object of $\C$, the tensor product $V \otimes (X, u)$ is a simple object of $\C^G$.

\medbreak An instance of an equivariantization exact sequence arises from exact sequences of finite groups: if $F$ is a finite group and $N \subseteq F$ is a normal subgroup, then the natural action of $F$ on finite-dimensional representations of $N$ induces an action the quotient group $G = F/N$ on $\Rep N$ by tensor autoequivalences. Moreover, there is an equivalence of tensor categories $\Rep F \cong (\Rep N)^G$, such that the forgetful functor $(\Rep N)^G \to \Rep N$ becomes identified with the restriction functor $\Rep F \to \Rep N$, $X \mapsto X_N$. (See \cite[Section 3]{ext-ty} for a discussion in the more general context of cocentral exact sequences of Hopf algebras.)

Thus we obtain that for any irreducible characters $\psi$ and $\chi$ of $F$ such that $N \subseteq \ker \psi$ and $\chi_N \in \Irr(N)$, the product $\psi\chi$ is again an irreducible character of $F$.
This well-known fact in finite group theory, is a consequence of a theorem of Gallagher; see for instance \cite[Corollary 6.17]{isaacs}.  \end{remark}

\section{The Frobenius-Perron graphs of an equivariantization}\label{graph-equiv}

Throughout this section $\C$ will be a fusion category and $G$ will be a finite group endowed with an action on $\C$ by tensor autoequivalences $\rho: \underline G \to \underline{\Aut}_\otimes(\C)$.

Let $\C^G$ be the equivariantization of $\C$ under the action $\rho$ and let $F: \C^G \to \C$ denote the forgetful functor.

\medbreak For every simple object $Z$ of $\C$, let $G_Z \subseteq G$ be the inertia subgroup of $Z$, that is, $G_Z = \{g\in G:\,  \rho^g(Z) \cong  Z \}$. Since $Z$ is simple, there is a $2$-cocycle $\alpha_Z: G_Z \times G_Z \to k^*$ defined by the relation
\begin{equation}\label{alfa} \alpha_Z(g, h)^{-1} \id_Z = c^g \rho^g(c^h)({\rho^{g, h}_{2_Z}})^{-1}(c^{gh})^{-1}: Z \to Z, \end{equation}
where, for all $g \in G_Z$, $c^g: \rho^g(Z) \to Z$ is a fixed isomorphism \cite[Subsection 2.3]{fusionrules-equiv}.

\medbreak Simple objects of $\C^G$ are parameterized by pairs $(Z, U)$, where $Z$ runs over the $G$-orbits on $\Irr(\C)$ and $U$ is an equivalence class of an  irreducible $\alpha_Z$-projective representation of $G_Z$.

We shall use the notation $S_{Z, U}$ to indicate the (isomorphism class of the) simple object corresponding to the pair $(Z, U)$. We have in addition
\begin{equation}\label{dim-equiv}\FPdim S_{Z, U} =  [G:G_Z] \dim U \FPdim Z.\end{equation}
Furthermore, $F(S_{Z, U})$ is a direct sum of conjugates of the simple object $Z$ in $\C$. See \cite[Corollary 2.13]{fusionrules-equiv}.

\medbreak \emph{We shall assume in what follows that the fusion category $\C$ is integral.} This is equivalent to the fusion category $\C^G$ being integral.

\begin{remark} The full subcategory of $\C^G$ whose simple objects are parameterized by $S_{\un, U}$, where $U$ runs over the nonisomorphic irreducible representations of $G$, is a fusion subcategory of $\C^G$ equivalent to $\Rep G$, that coincides with the kernel $\KER_F$ of the forgetful functor $F: \C^G \to \C$. In particular, $\Delta(G)$ is a subgraph of $\Delta(\C^G)$.
\end{remark}

\begin{lemma}\label{simple-relprime}  Let $X$ be a simple object of $\C^G$ such that $\FPdim X$ is relatively prime to the order of $G$. Then $F(X)$ is a simple object of $\C$.
\end{lemma}

\begin{proof} Let $(Z, U)$ be the pair corresponding to the simple object $X$, so that $X \cong S_{Z, U}$.
Since the number $[G:G_Z] \dim U$ divides the order of $G$, which by assumption is relatively prime to $\FPdim X$, then $[G:G_Z] \dim U  =1$ and $\FPdim F(X) = \FPdim X = \FPdim Z$.
Hence $F(X) = Z$ is a simple object of $\C$, as claimed.
\end{proof}

The following proposition will be used in the proof of our main results. It relies on a theorem of Michler \cite[Theorem 5.4]{michler} whose proof relies in turn on the classification of finite simple groups.

\begin{proposition}\label{ext-nonab} Let $G$ be a nonabelian finite group and let $q$ be a prime number. Suppose that $q$ is a vertex of the graph $\Delta(\C^G)$ such that $q \notin \pi(\cd G)$. Then one of the following holds:
\begin{enumerate}
\item[(i)] There exists $p \in \pi(\cd G)$ with $d(p, q) \leq 1$, or
\item[(ii)] $G$ has a nontrivial normal abelian Sylow subgroup.
\end{enumerate}
In particular, if $G$ does not satisfy (ii), then the number of connected components of $\Delta(\C^G)$ is at most equal to the number of connected components of $\Delta(G)$. \end{proposition}

\begin{proof} Let  $X \in \C^G$ be  a simple object such that $q$ divides $\FPdim X$.
Suppose that $\FPdim X$ is divisible by some prime factor $p$ of the order of $G$. Then either $p \in \pi(\cd G)$, in which case (i) holds, or, by Michler's Theorem \cite[Theorem 5.4]{michler}, $G$ has a normal abelian Sylow $p$-subgroup, and thus (ii) holds.

\medbreak Thus we may assume that $\FPdim X$ is relatively prime to the order of $G$. Hence $F(X)$ is a simple object of $\C$, by Lemma \ref{simple-relprime}.

On the other hand, since $G$ is  nonabelian, $\pi(\cd G) \neq \emptyset$. Let $p \in \pi(\cd G)$, and let $Y$ be a simple object of $\Rep G = \KER_F$  whose dimension is divisible by $p$.
It follows from Corollary \ref{gallagher-exact} that $Y\otimes X$ is a simple object of $\C^G$. Since $pq$ divides $\FPdim (Y \otimes X)$, then $d(q, p) = 1$ in $\Delta(\C^G)$. Therefore (i) holds in this case.

\medbreak Finally observe that, if $G$ has no nontrivial normal abelian Sylow subgroup, then every vertex of $\Delta(\C^G)$ is connected to some vertex of $\Delta(G)$, whence the statement on the number of connected components.   This finishes the proof of the theorem. \end{proof}

\subsection{Equivariantization under the action of a nonabelian group}
We continue to assume in this subsection that $G$ is a finite group and $\C$ is an integral fusion category endowed with an action by tensor autoequivalences $\underline{G} \to \underline{\Aut}_{\otimes}(\C)$.
We shall now apply the results obtained so far in order to study the prime graph of $\C^G$, where $G$ is  nonabelian.

\begin{remark}\label{rmk-action} We shall use repeatedly in what follows the following fact. Suppose $N$ is a normal subgroup of $G$. The action $\rho: \underline G \to \underline{\Aut}_\otimes(\C)$ induces, by restriction, an action of $N$ on $\C$ by tensor autoequivalences. Moreover, there is an action by tensor autoequivalences of quotient group $G/N$ on the equivariantization $\C^N$, such that  $\C^G$ is equivalent $(\C^N)^{G/N}$ as tensor categories.  See \cite[Corollary 5.4]{br-bu}, \cite[Example 2.9]{DMNO}.

Consider the fusion subcategory $\Rep N \subseteq \C^N$. It follows from the proof of \cite[Corollary 5.4]{br-bu} that the restriction of the $G/N$-action to the fusion subcategory $\Rep N \subseteq \C^N$ is equivalent to the adjoint action of $G$ on $\Rep N$, induced by the restriction functor $\Rep G \to \Rep N$.
\end{remark}

\begin{lemma}\label{ext-sol-p} Suppose that $G$ is a nonabelian finite $p$-group, where $p$ is a prime number.
Then $d(p, q) \leq 1$ for all vertex $q$ of the graph $\Delta(\C^G)$. In particular, the graph $\Delta(\C^G)$ is connected and its diameter is at most $2$. \end{lemma}

\begin{proof} A nonabelian $p$-group has no nontrivial abelian Sylow subgroups. Hence the statement follows immediately from Proposition \ref{ext-nonab}. \end{proof}

\begin{proposition}\label{equiv-nilp} Suppose that $G$ is a nonabelian nilpotent group. Then the graph $\Delta(\C^G)$ is connected.
\end{proposition}

\begin{proof} Since $G$ is nilpotent, then $G$ is the direct product of its Sylow subgroups. Because $G$ is not abelian, we may assume that there is a prime number $p$ such that the Sylow $p$-subgroup $S_p$ is not abelian. Let $L = G/S_p$, so that $G \cong S_p \times L$. Then, by Remark \ref{rmk-action}, the group $S_p$ acts by tensor autoequivalences on the equivariantization $\C^L$ and there is an equivalence of fusion categories $\C^G \cong (\C^L)^{S_p}$. Lemma \ref{ext-sol-p} implies the proposition.
\end{proof}

The proof of the following theorem follows the lines of the proof of an analogous result for the character degree graph of a finite group in \cite[Theorem 18.4]{manz-wolf}.

\begin{theorem}\label{equiv-solv} Suppose that the group
$G$ is  solvable nonabelian. Then the graph $\Delta(\C^G)$ has at most two connected components.
\end{theorem}

\begin{proof}
We consider first the case where $G/L$ is abelian, for every nontrivial normal subgroup $L \subseteq G$.
Note that this assumption is equivalent to the requirement that $G' \subseteq G$ is the unique minimal normal subgroup of $G$.

By Lemma \ref{ext-sol-p}, we may assume that $G$ is not of prime power order. Then it follows from \cite[Lemma 12.3]{isaacs}, that $G$ has a unique irreducible degree $d > 1$. Furthermore, $G$ is a Frobenius group with Frobenius kernel $G'$ which is an elementary abelian $r$-group, for some prime number $r$, and Frobenius complement $H$ which is abelian of order $d$.
Observe that $|G| = r^nd$, for some $n \geq 1$. In addition, the vertex set of $\Delta(G)$ coincides with the set $\pi(d)$ of prime divisors of $d$ and $\Delta(G)$ is the complete graph on the set $\pi(d)$.

\medbreak Let $p$ be a vertex of $\Delta(\C^G)$, and let $X$ be a simple object of $\C^G$ such that $p \vert \FPdim X$. If $r | \FPdim X$, then $r$ is also a vertex of $\Delta(\C^G)$ and $d(p, r) \leq 1$.

Assume that $\FPdim X$ is not divisible by $r$. We shall show that every prime divisor of $\FPdim X$ is connected by a path with every vertex of $\Delta(G)$.
Suppose first that $\FPdim X$ is not relatively prime to $d$, then it is divisible by some prime number $q$ in $\pi(\cd G)$ and  $d(p, q) \leq 1$. Since $\Delta(G)$ is a complete graph, this implies that $p$ is connected by a path with every vertex $t$ of $\Delta(G)$.

 Otherwise, $\FPdim X$ is relatively prime to the order of $G$. Hence $F(X)$ is a simple object of $\C$, by Lemma \ref{simple-relprime}. Let $Y \in \Irr(G)$ of dimension $d$. Corollary \ref{gallagher-exact} implies that $Y \otimes X$ is a simple object of $\C^G$. Since $\FPdim (Y \otimes X)$ is divisible by $pq$, for all prime number $q \in \pi(\cd G)$, we get that $d(p, q) \leq 1$, for all vertex $q$ of $\Delta(G)$. This shows that $\Delta(\C^G)$ has at most $2$ connected components in this case.

\medbreak We now consider the general case. Let $N \subseteq G$ be a maximal normal subgroup such that the quotient $L = G/N$ is nonabelian. As pointed out in Remark \ref{rmk-action}, the group $L$ acts on the equivariantization $\C^N$ by tensor autoequivalences in such a way that  $\C^G$ is equivalent $(\C^N)^L$ as tensor categories. On the other hand, the maximality of $N$ implies that every proper quotient of the group $L$ is abelian. Then, by the first part of the proof, the prime graph of $(\C^N)^L \cong \C^G$ has at most $2$ connected components. This finishes the proof of the theorem.
\end{proof}

We shall need the following lemma, which is contained in the proof of  \cite[Proposition 2]{manz-sw}.

\begin{lemma}\label{lema-msw} Let $G$ be a  finite group and let $q$ be a prime number. Suppose that $N$ is a nonsolvable normal subgroup of $G$ of $q$-power index and $Q \subseteq N$ is a normal $q$-Sylow subgroup of $N$. If $Q_0 \subseteq G$ is a $q$-Sylow subgroup such that every nonlinear irreducible representation of $N/Q$ is stable under $Q_0/Q$, then $Q_0$ is normal in $G$. \end{lemma}

\begin{proof} Since $N$ is normal in $G$ and $Q$ is a characteristic subgroup of $N$, then $Q$ is a normal $q$-subgroup of $G$.  In particular, $Q \subseteq Q_0$ for any $q$-Sylow subgroup $Q_0$ of $G$. Suppose first that $Q_0/Q$ does not centralize $N/Q$ in $G/Q$. Then $Q_0/Q$ acts nontrivially on $N/Q$ by conjugation. Furthermore, this is a coprime action in the sense of \cite{isaacs2} and every nonlinear irreducible representation of $N/Q$ is stable under $Q_0/Q$, by assumption. It follows from \cite[Theorem A]{isaacs2} that the derived group $(N/Q)'$ is nilpotent and in particular, $N/Q$ is solvable. Hence $N$ must be solvable as well, against the assumption.

Therefore we get that $Q_0/Q$ and $N/Q$ centralize each other in $G/Q$. This implies that $N \subseteq N_G(Q_0)$. Since $N$ has $q$-power index, then $G = N Q_0 \subseteq N_G(Q_0)$ and therefore $Q_0$ is normal in $G$, as claimed.
\end{proof}

\begin{theorem}\label{thm-leq3} Let $G$ be a nonabelian finite group and let $\C$ be an integral fusion category endowed with an action by tensor autoequivalences $\underline{G} \to \underline{\Aut}_{\otimes}(\C)$.  Then the graph $\Delta(\C^G)$ has at most three connected components.
\end{theorem}

\begin{proof} The proof is by induction on the order of $G$. We follow the lines of the last part of the proof of \cite[Proposition 2]{manz-sw}. By Theorem \ref{equiv-solv} we may assume that $G$ is not solvable. By Proposition \ref{ext-nonab} we may further assume that $G$ contains a nontrivial normal abelian Sylow subgroup; otherwise the result follows from the corresponding result for groups; see Theorem \ref{ppties-graphg-1} (i).

\medbreak Let  $S_{p_1}, \dots S_{p_r}$ be the nontrivial normal (not necessarily abelian) Sylow subgroups of $G$, where $p_1, \dots, p_r$ are pairwise distinct prime divisors of the order of $G$ and $|S_{p_i}| = |G|_{p_i}$, $i = 1, \dots, r$. Let also  $S$ be the subgroup of $G$ generated by $S_{p_1}, \dots, S_{p_r}$. Then $S$ is a normal subgroup of $G$ isomorphic to the direct product $S_{p_1} \times \dots \times S_{p_r}$. Thus $S \subsetneq G$, because $G$ is not solvable by assumption.

Let $N \subsetneq G$ be a maximal normal subgroup such that $S \subseteq N \subseteq G$. So that the quotient $G/N$ is simple and its order is relatively prime to $p_i$, for all $i =1, \dots, r$.

As in the proof of Theorem \ref{equiv-solv}, there is an action of $G/N$ on the equivariantization $\C^N$ such that  $\C^G \cong (\C^N)^{G/N}$.

Suppose first that $G/N$ is nonabelian. Proposition \ref{ext-nonab} implies that the number of connected components of the prime graph of $(\C^N)^{G/N} \cong \C^G$ is at most equal to the number of connected components of $\Delta(G/N)$. Hence this number is at most three, by Theorem \ref{ppties-graphg-1} (i).

\medbreak Therefore we may assume that $G/N$ is cyclic of prime order $q$. In particular, $N$ is not solvable. Note that $q \neq p_i$, for all $i = 1, \dots, r$.
Since $|N| < |G|$, we may inductively assume that the graph $\Delta(\C^N)$ has at most three connected components.

Simple objects of $(\C^N)^{G/N}$ are parameterized by pairs $S_{Y, U}$, where $Y$ runs over a set of representatives of the orbits of the action  of $G/N$ on $\Irr(\C^N)$ and $U$ is an irreducible projective representation of the stabilizer of $Y$. Since $G/N$ is cyclic of order $q$, then $G_Y = 1$ or $G/N$. Hence, in view of Formula \eqref{dim-equiv}, for each $Y \in \Irr(\C^N)$, we have two possibilities:
$$(a)\; \FPdim S_{Y, U} = q \FPdim Y, \quad \text{ or } \quad (b)\; \FPdim S_{Y, U} = \FPdim Y.$$
Possibility (b) holds if and only if $Y$ is $G/N$-stable.
Furthermore, we may assume that there is a simple object of $\C^N$ which is not $G/N$-stable: otherwise, $\cd(\C^G) = \cd(\C^N)$ and therefore $\C^G$ and $\C^N$ have the same prime graph. Hence the result follows in this case.

\medbreak We thus obtain that $\vartheta(\Delta(\C^G)) = \{q\} \cup \vartheta(\Delta(\C^N))$. Then it will be enough to show that $q$ is connected with a vertex in $\Delta(\C^N)$ and use the inductive assumption on $\C^N$.

We may assume that $N$ has no nonlinear irreducible representation of dimension divisible by $q$ (if such a representation exists, then $q$ is connected to some vertex of $\Delta(N)$ which is a subgraph of $\Delta(\C^N)$).    Therefore, by Michler's theorem \cite[Theorem 5.4]{michler}, $N$ has a normal abelian $q$-Sylow subgroup $Q$. Note that $\vartheta(\Delta(N)) = \vartheta(\Delta(N/Q))$, by Ito's theorem.

Similarly, if $N$ has a nonlinear irreducible representation $Y$ with  $G_Y = 1$, then $\FPdim S_{Y, U} = q\dim Y$, and $q$ is connected with the prime divisors of $\dim Y$, hence we are done.

\medbreak Therefore we may assume that every nonlinear irreducible representation of $N/Q$ is stable under $G/Q$. Lemma \ref{lema-msw} implies that $G$ has a normal $q$-Sylow subgroup $Q_0$. The choice of the normal subgroup $N$ implies that $Q_0 \subseteq N$ and, in particular, $[G: N]$ is relatively prime to $q$, which is a contradiction. This shows that $q$ must be connected to some vertex of $\Delta(\C^N)$ and, by induction, we get that the graph $\Delta(\C^G)$ has at most three connected components.
The proof of the theorem in now complete.
\end{proof}

\subsection{Equivariantizations of pointed fusion categories}

Throughout this subsection let $G$ be a finite group and let $\underline G \to \underline{\Aut}_\otimes(\C)$ be an action of $G$ by tensor autoequivalences on a pointed fusion category $\C$. These actions are described by Tambara in \cite[Section 7]{tambara}. Let us recall the description here.

\medbreak Since $\C$ is pointed, we may assume that $\C = \C(A, \omega)$, where $A$ is the group of isomorphism classes of invertible objects in $\C$, and  $\omega\in H^3(A, k^*)$ is an invertible normalized 3-cocycle. Recall that $\C(A, \omega) = \vect^A_\omega$ is the category of finite dimensional $A$-graded vector spaces with associativity constraint induced by $\omega$.

\medbreak Any action $\rho: \underline G \to \underline \Aut_{\otimes} \C$ is determined by an action by group automorphisms of $G$ on $A$,  $x \mapsto {}^ga$, $a \in A$, $g \in G$, and two maps $\tau: G \times A \times A \to k^*$, and $\alpha: G \times G \times A \to k^*$, satisfying
\begin{align}\label{uno}\frac{\omega(a, b, c)}{\omega({}^ga, {}^gb, {}^gc)} & = \frac{\tau(g; ab, c) \, \tau(g; a, b)}{\tau(g; b, c) \, \tau(g; a, bc)} \\
\label{dos} 1 & = \frac{\alpha(h, l; a) \, \alpha(g, hl; a)}{\alpha(gh, l; a) \, \alpha(g, h; {}^la)} \\
\label{tres}\frac{\tau(gh; a, b)}{\tau(g; {}^ha, {}^hb) \, \tau(h; a, b)} & = \frac{\alpha(g, h; a) \, \alpha(g, h; b)}{\alpha(g, h; ab)},
\end{align} for all $a, b, c \in A$, $g, h, l \in G$.

Let us assume that, in addition, $\tau$ and $\alpha$ are normalized such that $\tau(g; a, b) = \alpha(g, h; a) = 1$, whenever some of the arguments $g$, $h$, $a$ or $b$ is an identity.

\medbreak Let us denote $\alpha_a(g, h): = \alpha(g, h; a)$ and $\tau_{a, b}(g): = \tau(g; a, b)$, $a, b \in A$, $g, h \in G$.
It follows from \eqref{dos} that, for all $a \in A$, $\alpha_a: G_a \times G_a \to k^*$ is a 2-cocycle.
Let $[\alpha_a] \in H^2(G_a, k^*)$ denote the cohomology class of $\alpha_a$.

Let $A^G = \{ a \in A: \, G_a = G\}$ be the subgroup of fixed points of $A$ under the action of $G$. Then we have

\begin{lemma}\label{group-homo} The assignment $a \mapsto [\alpha_a]$ defines a group homomorphism $A^G \to H^2(G, k^*)$.
\end{lemma}

\begin{proof} Suppose $a, b \in A^G$. Then condition \eqref{tres} writes as follows:
$$(\alpha_a\alpha_b\alpha_{ab}^{-1}) (g, h) = \tau_{a, b}(gh)\tau_{a, b}^{-1}(g)\tau_{a, b}^{-1}(h) = d\tau_{a, b}(g, h),$$
for all $g, h \in G$, where $\tau_{a, b}: G \to k^*$ is the 1-cochain $g \mapsto \tau_{a, b}(g)$.
We obtain from this that $[\alpha_{ab}] = [\alpha_a] \, [\alpha_b] \in H^2(G, k^*)$, which proves the claim. \end{proof}

\medbreak The set of isomorphism classes of simple objects of $\C^G$ is parameterized by the isomorphism classes of simple objects $X_{a, U_a}$, where $a$ runs over the orbits of the action of $G$ on $A$ and $U_a$ is an irreducible projective representation of the stabilizer $G_a \subseteq G$ with factor set $\alpha_a$.  We shall say that the simple object $X_{a, U_a}$ \emph{lies over} $a \in A$. For every pair $(a, U_a)$ we have
\begin{equation}\label{dim-ep}\FPdim X_{a, U_a} = |^Ga| \, \dim U_a.\end{equation}

\begin{theorem}\label{eq-pted} The graph $\Gamma(\C(A, \omega)^G)$ has at most $3$ connected components. \end{theorem}

\begin{proof} In view of Theorem \ref{thm-leq3} and Remark \ref{graphs-sets}, we may assume that the group $G$ is abelian.

As a consequence of \cite[Corollary B]{ipgraph} (see Example \ref{orbits}), it follows from formula \eqref{dim-ep} that the Frobenius-Perron dimensions of simple objects $X_{a, U_a}$ lying over non-fixed points $a \in A\backslash A^G$ fall into at most two connected components of $\Gamma(\C(A, \omega)^G)$.

Since the group $G$ is abelian, the remaining nontrivial Frobenius-Perron dimensions of $\C(A, \omega)^G$ correspond to simple objects $X_{a, U_a}$ lying over elements $a \in A^G$ such that $[\alpha_a] \neq 1$, and for such simple objects we have $\FPdim X_{a, U_a} = \dim U_a$. Let $n_a$ denote the order of the class $[\alpha_a]$ in $H^2(G, k^*)$. Then $n_a$ divides $\dim U_a$. (Note in addition that, since $G$ is abelian, the dimensions of the irreducible projective representations $U_a$ are all equal.)

Suppose that $a, b \in A^G$ are such that $[\alpha_a], [\alpha_b] \neq 1$ and $n_a$ is relatively prime to $n_b$. In particular, $ab \in A^G$ and from Lemma
\ref{group-homo} we get that the order of the class $[\alpha_{ab}] = [\alpha_a] [\alpha_b]$ equals the product of the orders $n_an_b$. Hence the Frobenius-Perron dimensions of $X_{a, U_a}$ and $X_{b, U_b}$ are both connected to the Frobenius-Perron dimension of $X_{ab, U_{ab}}$, for any irreducible $\alpha_{ab}$-projective representation $U_{ab}$ of $G$. Therefore the nontrivial dimensions of simple objects of $\C(A, \omega)^G$ lying over fixed points of $A$ belong to the same connected component of $\Gamma(\C(A, \omega)^G)$. Thus this graph has at most three connected components, as claimed.
\end{proof}

\section{The graph of the representation category of a twisted quantum double}\label{s-tqd}
Let $G$ be a finite group and let $\omega$ be a 3-cocycle on $G$. Consider  the \emph{twisted quantum double} $D^{\omega}(G)$ of $G$ \cite{dpr}. The Frobenius-Perron graphs of the fusion category $\C = \Rep D^\omega(G)$ will be indicated by $\Delta(D^\omega(G))$, $\Gamma(D^\omega(G))$ in what follows. Similarly, the vertex set of the graph $\Delta(D^\omega(G))$ will be denoted by $\vartheta(D^\omega(G))$.

\medbreak The irreducible representations of $D^{\omega}(G)$ are parameterized by pairs $(a, U)$, where $a$ is a conjugacy class representative of $G$ and $U$ is a projective irreducible representation of $C_G(a)$, the centralizer of $a$ in $G$, with 2-cocycle $\alpha_{a}$ with coefficients in $k^{*}$ given by:
\begin{equation}\label{theta}\alpha_{a}(x,y) =
\frac{\omega(a,x,y)\omega(x,y,a)}{\omega(x,a,y)}, \quad x,y \in C_G(a).\end{equation}

The dimension of the representation $W_{(a, U)}$ corresponding to the pair $(a, U)$ is $\dim W_{(a, U)} = [G : C_G(a)] \dim U = |^Ga| \dim U$, where $^Ga \subseteq G$ denotes the conjugacy class of $a$.

\begin{remark}\label{tdd-equiv} Let $\C(G, \omega)$ denote the fusion category of finite-dimensional $G$-graded vector spaces with associativity given by $\omega$. The adjoint action of $G$ on itself gives rise to an action by tensor autoequivalences $G \to \underline{\Aut}_\otimes(\C(G, \omega))$ in such a way that $\Rep D^\omega(G) \cong \C(G, \omega)^G$.  \end{remark}

\medbreak  For each $a\in G$, consider the associated 2-cocycle $\alpha_a$ of $C_G(a)$ given by \eqref{theta}. The map $D_a:\omega\mapsto\alpha_a$ induces a group homomorphism $H^{3}(G,k^{*})\to H^{2}(C_C(a),k^{*})$. Following \cite{mason-ng} we will denote by $H^{3}(G,k^{*})_{ab}$ the subgroup $$H^{3}(G,k^{*})_{ab} = \bigcap_{a\in G}\ker D_a \subseteq H^{3}(G,k^{*}).$$

Let $p$ be a prime number. We shall denote by $\omega_p$ the $p$-part of the cocycle $\omega$. Thus, $\omega = \prod_{p \vert |\omega|} \omega_p$.
Let $S_\omega$ be the set of prime numbers defined as follows:
$$S_\omega=\{p \textrm{ prime: } \omega_p \notin H^{3}(G, k^*)_{ab}, \textrm{ and } G \textrm{ has a central Sylow $p$-subgroup}\}.$$
Note that  the order of $\omega$ is divisible by all primes $p \in S_\omega$.

\begin{lemma}\label{cor-sp} Let $p$ be a prime number and assume that $G$ has a central Sylow subgroup $S$. Then $p \in S_\omega$ if and only if $\omega\vert_S \notin H^3(S, k^*)_{ab}$.
\end{lemma}

\begin{proof} The assumption implies that $S$ is an abelian direct factor of $G$, that is, $G =S \times \tilde G$, where the order of $\tilde G$ is not divisible by $p$. Then there are 3-cocycles $\omega_{1} \in H^3(S, k^*)$ and $\omega_{2}  \in H^3(\tilde G, k^*)$ such that $\Infl\omega_{1} \Infl\omega_{2}=\omega$ \cite{brown}.

Observe that the order of the cohomology class of $\Infl\omega_{2}$ divides the order of $\tilde G$ and thus it is relatively prime to $p$. Therefore the class of  $\Infl\omega_{1}$ coincides with the $p$-part $\omega_p$ of the class of $\omega$. In addition $\omega|_S = \omega_p|_S = \omega_{1}$.

Suppose $\omega_p \in H^3(G, k^*)_{ab}$. Then,
for all $a \in S \subseteq Z(G)$, there is a $1$-cochain $f_a:G \to k^*$ such that $\alpha_a = df_a$, where $\alpha_a \in H^2(G, k^*)$ is associated to $\omega_p$ by \eqref{theta}. Letting $\alpha^1_a \in H^2(S, k^*)$ be the 2-cocycle associated to $\omega_1 = \omega_p\vert_S$, we have that $\alpha^1_a = df_a\vert_S$ and hence $\omega_1 \in H^3(S, k^*)_{ab}$.

Conversely, suppose that $\omega_1 = \omega\vert_S \in H^3(S, k^*)_{ab}$. Then for all $a \in S$, there exists $f_a : S \to k^*$ such that $\alpha'_a = df_a$, where $\alpha'_a \in H^2(S, k^*)$ is associated to $\omega\vert_S$ by \eqref{theta}. Let $\tilde f_{(a, h)}:G = S \times \tilde G \to k^*$ be the $1$-cochain defined by $\tilde f_{(a, h)}(b, g) = f_a(b)$, for all $a, b\in S$, $h, g \in \tilde G$. Letting $\alpha_{(a, h)}$ be the $2$-cocycle on $G = C_G(a)$ associated to $\omega_p = \Infl(\omega_1)$ by \eqref{theta}, we get that  $\alpha_{(a, h)} =  d\tilde f$.
Hence  $\omega_p = \Infl \omega_1 \in H^3(G, k^*)_{ab}$. This finishes the proof of the lemma.
\end{proof}

\begin{proposition}\label{contain1} $\Delta(G)$ and $\Delta^{\prime}(G)$ are subgraphs of $\Delta(D^{\omega}(G))$. Moreover,
$$\pi(\cs G) \subseteq \vartheta(D^{\omega}(G)) \subseteq \pi(\cs G) \cup S_\omega.$$\end{proposition}

\begin{proof}  Let $p$ be a vertex of the graph $\Delta'(G)$. Then there exists $a\in G$ such that $p$ divides $|^Ga|$.
Now, let $W$ be an irreducible representation of $D^{\omega}(G)$ corresponding to a pair  $(a, U)$, where $U$ is any projective irreducible representation of $C_G(a)$ with 2-cocycle $\alpha_{a}$. Then we have  $\dim (a, U)=|^Ga|\dim U$, and thus $p$ divides $\dim W$. Hence $p\in\Delta(D^{\omega}(G))$.
This shows that $\pi(\cs G) \subseteq \vartheta(D^{\omega}(G))$. Clearly, $\Delta^{\prime}(G) \subseteq \Delta(D^{\omega}(G))$ is a subgraph.

In view of Remark \ref{tdd-equiv},  $\Rep D^\omega(G) \cong \C(G, \omega)^G$ is an equivariantization. In particular, $\Rep G$ is equivalent to a fusion subcategory of $\Rep D^\omega(G)$ and therefore $\Delta(G) = \Delta(\Rep G)$ is also a subgraph of $\Delta(D^{\omega}(G))$.

\medbreak Assume that $p\in\vartheta(D^{\omega}(G))$  is such that $p\notin\pi(\cs G)$. Then for every $g\in G$ we have that $p$ does not divide $|^Gg|$.
It follows from \cite[Corollary 4]{camina} that the Sylow $p$-subgroup of \emph{G} is an abelian direct factor of \emph{G}. Then $G$ has a (unique) central Sylow $p$-subgroup $S_p$.

We may thus write $G\cong S_{p} \times \widetilde{G}$ with $(|S_{p}|, |\widetilde{G}|)$ = 1 and $S_{p}$ an abelian $p$-group. In view of \cite[Proposition 4.2]{mason-ng} we have the isomorphism of braided fusion categories
\begin{equation}\label{iso-domega}D^{\omega}(G) \cong D^{\omega_{1}}(S_{p})\otimes D^{\omega_{2}}(\widetilde{G}),\end{equation} where the 3-cocycles $\omega_{1}$ and $\omega_{2}$ satisfy $\Infl\omega_{1} \Infl\omega_{2}=\omega$.
Therefore the class of  $\Infl\omega_{1}$ coincides with the $p$-part $\omega_p$ of the class of $\omega$ and $\omega_p\vert_S =\omega_1$.

The isomorphism \eqref{iso-domega} implies that the irreducible representations of $\emph{D}^{\omega}(G)$ are given by tensor products of irreducible representations of $D^{\omega_{1}}(S_{p})$ and $D^{\omega_{2}}(\widetilde{G})$.

The number $p$ does not divide the order of $\widetilde{G}$. Then $p$ does not divide the dimension of any irreducible representation of $D^{\omega_{2}}(\widetilde{G})$. Hence $p$ divides $\dim Z$, where $Z$ is an irreducible representation of $D^{\omega_{1}}(S_{p})$. This implies that $D^{\omega_{1}}(S_{p})$ is not a commutative algebra. Then using \cite[Corollary 3.6]{mason-ng} we have $\omega_{1}\notin H^{3}(S_{p},k^{*})_{ab}$. Hence $\omega_p \notin H^3(G, k^*)_{ab}$, by Lemma \ref{cor-sp}.
 This proves that $p\in S_\omega$ and the result follows.
\end{proof}

\begin{proposition}\label{connected1} Let $p \in \vartheta(D^\omega(G)) \backslash \, \pi(\cs G)$. Then $p$ is connected with every vertex $q \in \Delta'(G)$.\end{proposition}

\begin{proof}
It follows from Proposition \ref{contain1} that  $p \in S_\omega$ and therefore the Sylow $p$-subgroup $S_{p}$ of $G$ is an abelian direct factor of $G$. Hence there is an isomorphism
$D^{\omega}(G) \cong D^{\omega_{1}}(S_{p})\otimes D^{\omega_{2}}(\widetilde{G})$, where $G = S_p \times \widetilde G$ and $[\Infl\omega_{1}] [\Infl\omega_{2}]=[\omega]$. Therefore the irreducible representations of $\emph{D}^{\omega}(G)$ are given by tensor products of irreducible representations of $D^{\omega_{1}}(S_{p})$ and $D^{\omega_{2}}(\widetilde{G})$.

Since $p \in \vartheta(D^\omega(G))$, and $p$ does not divide the dimension of any irreducible representation of $D^{\omega_{2}}(\widetilde{G})$, then $p$ divides $\dim Z$, for some  irreducible representation $Z$ of $D^{\omega_{1}}(S_{p})$.

Let $q \in \pi(\cs G)$. Observe that $\cs(G) = \cs (\widetilde G)$, because $G = S_p \times \widetilde G$ and $S_p$ is abelian.   Hence, there exists $a \in \widetilde G$ such that $q$ divides $|^{\widetilde G}a|$. Then $q$ divides the dimension of some irreducible representation $\widetilde W_{(a, U)}$ of $D^{\omega_{2}}(\widetilde{G})$. This implies that $pq$ divides the dimension of $Z \otimes \widetilde W_{(a, U)}$ and, since this is an irreducible representation of $D^\omega(G)$, then  $p$ and $q$ are connected by an edge in $\Delta(D^\omega(G))$.
\end{proof}

Using the previous propositions, we obtain the following theorem about the graph $\Delta(D^\omega(G))$.

\begin{theorem}\label{tqd}
Let $G$ be a finite group and let $\omega$  be a 3-cocycle on $G$. Then the graph $\Delta(D^\omega(G))$ has at most two connected components and its diameter is at most $3$. Furthermore, we have:
\newcounter{lista1}
\begin{enumerate}
\item[(i)] Suppose $G$ is nilpotent. Then $\Delta(D^\omega(G))$ is the complete graph on the vertex set $S_\omega$.
  \item[(ii)] If the set $S_\omega$  is non-empty, then the graph $\Delta(D^\omega(G))$ is connected and its diameter is at most two.
  \item[(iii)] Suppose that $\Delta(D^\omega(G))$ is not connected. Then $S_\omega = \emptyset$ and $G$ is a quasi-Frobenius group with abelian complement and kernel. Furthermore, we have $\Delta(D^\omega(G)) = \Delta'(G)$ has two connected components and each of them is a complete graph.
\item[(iv)] Suppose that $G$ is not solvable. Then $\Delta(D^\omega(G))$ is connected and its diameter is at most $2$. \end{enumerate}
 \end{theorem}

\begin{proof} (i). The assumption implies that $G$ is the direct product of its Sylow subgroups: $G\cong S_{p_1} \times \dots \times S_{p_n}$, where $p_1, \dots, p_n$ is the set of prime divisors of the order of $G$. By \cite[Proposition 4.2]{mason-ng}, we have
$$D^{\omega}(G) \cong D^{\omega_{1}}(S_{p_1})\otimes \dots \otimes D^{\omega_{n}}(S_{p_n}),$$ where the 3-cocycles $\omega_{i} \in H^3(S_{p_i}, k^*)$ satisfy $\Infl\omega_{1} \dots \Infl\omega_{n}=\omega$. Thus in this case $\Rep D^\omega(G)$ is a nilpotent braided category and part (i) follows from Proposition \ref{nilp}. Observe that the vertex set  coincides in this case with the set $S_\omega$.

\medbreak (ii). We may assume that the group $G$ is not abelian.
Let $p \in \vartheta(D^\omega(G))$  be such that $p \in S_\omega$. Using Proposition \ref{connected1}, for any $q \in \vartheta(D^\omega(G))$ such that $q \in \Delta'(G)$ we have that $p$ and $q$ are joined by an edge. This proves that the graph is connected. If $s$ and $r$ are two vertex in $\Delta'(G)$ then either $s$ and $r$ and joined by an edge or there exists a path of length $2$, which consists of the edges $(s, p)$ and $(p, r)$.

Similarly, suppose that $p \neq l \in S_\omega$ is another vertex. Since $G$ is not abelian, then  there exists a vertex $t \in \Delta'(G)$. Hence $t$ is connected to both $p$ and $l$ and therefore there is a path $(p, t)$,  $(l, t)$ from $p$ to $l$, of length $2$. It follows that the graph $\Delta(D^\omega(G))$ has diameter at most $2$. This proves (ii).

\medbreak (iii). Suppose that $\Delta(D^\omega(G))$ is not connected. Then $S_\omega = \emptyset$, by part (ii). It follows from Proposition \ref{contain1} that $\vartheta(D^\omega(G)) = \pi(\cs G)$ and, since the graph $\Delta'(G)$ is a subgraph of $\Delta(D^\omega(G))$, it is not connected neither. Theorem \ref{ppties-graphg} (ii) implies that $G$ is a Frobenius group with abelian kernel and complement. The disconnectedness assumption on the graph $\Delta(D^\omega(G))$ entails that there cannot be more edges in $\Delta(D^\omega(G))$ than there are in $\Delta'(G)$. Hence $\Delta(D^\omega(G)) = \Delta'(G)$ and we get (iii).

\medbreak (iv). By (ii), we may assume that the set $S_\omega$ is empty, and thus $\vartheta(D^\omega(G)) = \pi(\cs G)$, by Proposition \ref{contain1}. Since $G$ is not solvable, we have that $\Delta'(G)$ is connected is its diameter is at most $2$, by Theorem \ref{ppties-graphg} (iii). Hence the same holds for $\Delta(D^\omega(G))$. This shows (iv).

\medbreak The statement on the number of components of $\Delta(D^\omega(G))$ follows from (iii). Let us prove the statement on the diameter. We have that $\vartheta(D^{\omega}(G)) \subseteq \pi(\cs G) \cup S_\omega$. If $S_\omega \neq \emptyset$, then the diameter of $\Delta(D^\omega(G))$ is at most $2$, in view of part (ii). Suppose, on the other hand, that $S_\omega = \emptyset$. Since $\Delta'(G)$ is a subgraph of $\Delta(D^\omega(G))$,  the diameter of $\Delta(D^\omega(G))$ is not bigger than the diameter of $\Delta'(G)$. By Theorem \ref{ppties-graphg}, the diameter of $\Delta'(G)$ is at most $3$. This finishes the proof of the theorem.
\end{proof}

\begin{example} Let $G$ be a Frobenius group with abelian kernel $N$ and abelian complement $H$ of orders $n$ and $m$, respectively. The conjugacy class sizes of $G$ are $n$ and $m$. The Frobenius kernel $N$ is determined by $H$ by the relation
$$N = (G \backslash \cup_{g \in G}{}^g\!H) \cup \{ e\}.$$ See \cite{isaacs}. It follows that if $a \in G$, then either $a$ is conjugate to some element of $H$ or $a \in N$. Hence the conjugacy classes of elements of $G$ are those of the form $^Ga$, where $a \in H$ or $a \in N$.

Since the groups $N$ and $H$ are abelian by assumption, we obtain that $C_G(a) = H$ if $e \neq a \in H$,  and $C_G(a) = N$ if $e \neq a \in N$. In addition, every irreducible representation of $C_G(a)$ is one-dimensional in either of these cases. As for the identity element $e$, the irreducible degrees of the centralizer $C_G(e) = G$ divide the order of $H$, by Ito's theorem.
Hence the graph $\Delta(D(G))$ coincides with the graph $\Delta'(G)$.

\medbreak Take for instance $G$ to be the Frobenius group $G = \mathbb Z_{341} \rtimes \mathbb Z_{10}$ (with the action of $\mathbb Z_{10}$ on $\mathbb Z_{341}$ such that the generator acts by an automorphism of order $10$). Let $\C = \Rep D(G)$. Then $\C$ is a solvable fusion category such that $\Delta(\C)$ is the graph
$$\begin{CD}\bullet \overline{\qquad } \bullet \quad \bullet \overline{\qquad } \bullet.\end{CD}$$
Observe that this graph is not allowed for $\Delta(G)$, if $G$ is a solvable group; see \cite[Theorem 4.3]{lewis}.
\end{example}

\section{Braided fusion categories}\label{braided}

We begin this section by recalling some definitions and properties of braided fusion categories. We refer the reader to \cite{BK}, \cite{DGNOI} for a detailed exposition.

\medbreak Recall that a fusion category $\C$ is called \emph{braided} if it is endowed with a natural isomorphism, called a \emph{braiding},
$$c_{X, Y}: X \otimes Y \to Y \otimes X, \quad X, Y \in \C,$$ subject to the hexagon axioms.

\medbreak A braided fusion category $\C$ is called \emph{symmetric} if $c_{Y, X}c_{X, Y} = \id_{X \otimes Y}$, for all objects $X, Y \in \C$.

If $G$ is a finite group, then the category $\Rep G$ is symmetric when endowed with the standard braiding (given by the flip isomophism). A symmetric fusion category $\C$ is \emph{Tannakian} if $\C \cong \Rep G$ as braided fusion categories, for some finite group $G$.

The category $\svect$ of finite dimensional super vector spaces $V = V_0 \oplus V_1$ is a pointed symmetric fusion category which is not Tannakian. The braiding $c: V \otimes W \to W \otimes V$, between two objects $V, W \in \svect$ is defined by $c(v \otimes w) =  (-1)^{ab} w\otimes v$, for homogeneous elements $v \in V_a$, $w \in W_b$, $0 \leq a, b \leq 1$.

\medbreak Every symmetric fusion category $\C$ is \emph{super-Tannakian}, that is, there exist a finite group
$G$ and a central element $u \in G$ of order 2, such that $\C$ is equivalent  as a braided tensor category to the category $\Rep(G, u)$ of representations of $G$ on finite-dimensional super-vector spaces where $u$ acts as the parity operator \cite{deligne}.

If $\C$ is a symmetric fusion category, there is a canonical $\mathbb Z_2$-grading on $\C$, $\C = \C_1 \oplus \C_{-1}$, induced by the unique positive spherical structure on $\C$.   The category $\C_1 \cong \Rep G/(u)$ is the maximal Tannakian subcategory of $\C$. Furthermore, if $\theta$ denotes the corresponding balanced structure, then $\C_1$ coincides with the full subcategory of objects $X$ such that $\theta_X = \id_X$. See \cite[Corollary 2.50]{DGNOI}.

\medbreak Let $\C$ be a braided fusion category. Two simple objects $X$ and $Y$ of $\C$ are said to \emph{centralize each other} if $c_{Y, X}c_{X, Y} = \id_{X \otimes Y}$.
The  \emph{M\" uger centralizer}  of a fusion subcategory $\D$, indicated by $\D'$, is the full fusion
subcategory generated by all objects $X \in \C$ such that $c_{Y, X}c_{X, Y} =
\id_{X \otimes Y}$, for all objects $Y \in \D$. The centralizer $\C'$ of $\C$ is a symmetric subcategory, called the \emph{M\" uger center} of $\C$. We have that $\C$ is symmetric if and only if $\C' = \C$. If  $\C' \cong \vect$, then $\C$ is called \emph{non-degenerate}, and it is called \emph{slightly degenerate} if $\C' \cong \svect$.

\medbreak More generally, $X$ and $Y$ are said to \emph{projectively centralize each other} if $c_{Y, X}c_{X, Y}$ is a scalar multiple of the identity map of $X \otimes Y$. By \cite[Proposition 3.22]{DGNOI}, this condition is equivalent to requiring that $X$ centralizes $Y \otimes Y^*$ or also that $Y$ centralizes $X \otimes X^*$.

The projective centralizer of a simple object $X$ (respectively, of a full subcategory $\D$) is the full subcategory of $\C$ whose objects projectively centralize $X$ (respectively, every simple object of $\D$). This is a fusion subcategory of $\C$ \cite[Lemma 3.15]{DGNOI}.

\subsection{Tannakian subcategories of braided fusion categories}\label{tann-subc}

Let $\C$ be a brai\-ded fusion category. Suppose $\E$ is any Tannakian subcategory of $\C$, and let $G$ be a finite group such that $\E \cong \Rep G$ as symmetric categories. The algebra $A = k^G$ of functions on
$G$ with the regular action of $G$ is a commutative algebra in $\E$.

The fusion category $\D = \C_G$ of right $A$-modules in $\C$ is called the de-equivarianti\-zation of
$\C$ with respect to $\Rep G$, and we have that $\C \cong \D^G$ is an equivariantization. The category $\D$ is a braided $G$-crossed fusion category \cite{turaev, turaev2}. That is, $\D$ is endowed with a $G$-grading $\D
= \oplus_{g \in G}\D_g$ and an action of $G$ by tensor autoequivalences
$\rho:\underline G \to \underline \Aut_{\otimes} \, \D$, such that $\rho^g(\D_h) \subseteq
\D_{ghg^{-1}}$, for all $g, h \in G$, and a $G$-braiding $\sigma: X \otimes Y \to
\rho^g(Y) \otimes X$, $g \in G$, $X \in \D_g$, $Y \in \D$, subject to appropriate compatibility conditions.
The neutral component $\C_G^0$ of $\D =\C_G$ with respect to the associated $G$-grading is a braided fusion category.

We have that $\C$ is non-degenerate if and only if $\C_G^0$ is non-degenerate and the $G$-grading on $\D$ is faithful. In this case there is an equivalence of braided fusion categories $$\C \boxtimes (\C_G^0)^{\rev} \cong \Z(\C_G).$$
See \cite{mueger-gcrossed}, \cite[Corollary 3.30]{DMNO}.

\medbreak Conversely, if $\D$ is any $G$-crossed braided fusion category, then
the equivariantization $\C = \D^G$ under the action of $G$ is a braided fusion category, and $\C$ contains $\E \cong \Rep G$
as a Tannakian subcategory, under  the canonical embedding $\Rep G \to \D^G \cong \C$.

\medbreak Assume in addition that the Tannakian subcategory $\E$ is contained in $\C'$. Then the braiding of $\C$ induces a braiding in the de-equivariantization $\C_G$ such that the canonical functor $\C \to \C_G$ is a braided tensor functor and the action of $G$ on $\C_G$ is by braided autoequivalences \cite{bruguieres}, \cite{mueger-galois}, \cite[Corollary 5.31]{tensor-exact}.

\begin{remark}\label{ext-sym} Suppose $\C \cong \Rep(G, u)$ is a symmetric fusion category. So  that $\C_1 \cong \Rep G/(u)$ is its maximal Tannakian subcategory.  It follows from the previous discussion that the de-equivariantization $\D = \C_{G/(u)}$ is a braided fusion category endowed with an action of $G/(u)$ by braided autoequivalences such that $\C \cong \D^{G/(u)}$. Observe in addition that, if $\C$ is not Tannakian, then we have $\D \cong \svect$ as braided fusion categories. In fact, we have $\FPdim \D = \FPdim \C / \FPdim \C_1 = 2$ and, since the canonical functor $\C \to \D$ is a braided tensor functor, then $\D$ is symmetric (and not Tannakian). As a consequence, in the case where  $\C$ is not Tannakian, we have an exact sequence of braided fusion categories $$\Rep G/(u) \to \C \to \svect.$$
\end{remark}

\subsection{$S$-matrix and Verlinde formula for modular categories}

Recall that a \emph{premodular} category is a braided fusion category equipped with a spherical structure. Equivalently, $\C$ is a braided fusion category endowed with a \emph{balanced structure}, that is, a natural automorphism $\theta: \id_\C \to \id_\C$ satisfying
\begin{equation}\label{bal}\theta_{X \otimes Y} = (\theta_X \otimes \theta_Y) c_{Y, X}c_{X, Y}, \quad \theta_X^* = \theta_{X^*},\end{equation}
for all objects $X, Y$ of $\C$ \cite{bruguieres} (see also \cite[Subsection 2.8.2]{DGNOI}).

\medbreak Suppose $\C$ is a premodular category. Let $X, Y$ be simple objects of $\C$ and let $S_{X, Y} \in k$ denote the quantum trace of the squared braiding $c_{Y, X}c_{X, Y}:X \otimes Y \to Y \otimes X$.

The $S$-matrix of $\C$ is defined in the form $S = (S_{XY})_{X, Y \in \Irr(\C)}$.
This is a symmetric matrix with entries in a cyclotomic field that satisfies  $$S_{XY^*} = \overline{S_{XY}} = S_{X^*Y},$$ for all $X, Y \in \Irr(\C)$, where $\overline{S_{XY}}$ denotes the complex conjugate. In particular, the absolute value of $S_{XY}$ is determined by  $$|S_{XY}|^2 = S_{XY}S_{XY^*} = S_{XY}S_{X^*Y}, \quad X, Y \in \Irr(\C).$$

\medbreak The premodular  category $\C$ is called \emph{modular} if the $S$-matrix is non-degenerate \cite{turaev-b}. Equivalently, $\C$ is modular if and only if it is non-degenerate \cite[Proposition 3.7]{DGNOI}.

For every $X, Y, Z \in \Irr(\C)$, let $N_{XY}^Z = \dim \Hom_\C(Z, X \otimes Y)$ be the multiplicity of the simple object $Z$ in the tensor product $X \otimes Y$.
Suppose $\C$ is modular. Then the following relation, known as \emph{Verlinde formula}, holds:
\begin{equation}\label{verlinde}N_{XY}^Z = \frac{1}{\dim \C}\sum_{T \in \Irr(\C)} \frac{S_{XT} \, S_{YT} \, S_{Z^*T}}{d_T},\end{equation}
for all $X, Y, Z \in \Irr(\C)$, where $d_T$ denotes the categorical dimension of the object $T$ and $\dim \C = \sum_{T \in \Irr(\C)} d_T^2$ is the categorical dimension of $\C$. See \cite[Theorem 3.1.14]{BK}.

\medbreak Let $T$ be a simple object of $\C$. By \cite[Lemma 6.1]{gel-nik},
$|S_{XT}| = |d_X d_T|$, for all $X \in \Irr(\C)$, if and only if $T$ is an invertible object of $\C$.

\medbreak \emph{For the rest of this section,  $\C$ will be a weakly integral braided fusion category over $k$.} That is, $\FPdim \C$ is a natural integer.
We shall consider the category $\C$ endowed with the canonical positive spherical structure with respect to which categorical dimensions of simple objects coincide with their Frobenius-Perron dimensions \cite[Proposition 8.23]{ENO}. The corresponding balanced structure will be denoted by $\theta: \id_\C \to \id_\C$.

\subsection{De-equivariantization by the maximal Tannakian subcategory of the M\" uger center and fusion rules} Suppose that $\E \cong \Rep G$ is a Tannakian subcategory of $\C$ contained in $\C'$. This amounts to the assumption that $\theta_X = \id_X$, for every object $X$ of $\E$. Consider the de-equivariantization $\C_G$ of $\C$, which is also weakly integral and therefore a premodular category with canonical balanced structure $\overline \theta$. The action by braided autoequivalences of $G$ on $\C_G$ and the canonical functor $\C \to \C_G$  preserve the balanced structures, that is,
$$\overline\theta_{F(X)} = F(\theta_X), \quad \overline\theta_{\rho^g(Y)} = \overline\theta_Y,$$
for all objects $X$ of $\C$ and $Y$ of $\C_G$, and for all $g \in G$. See \cite[Lemme 3.3]{bruguieres}.

\begin{lemma}\label{fermionic} Let  $Z \in \C'$ be a simple object. Suppose that $Z$ is a constituent of $X \otimes X^*$ for some simple object $X$ of $\C$. Then $\theta_Z = 1$.
\end{lemma}

In the case where $Z$ is an invertible object of $\C$, the lemma is contained in \cite[Lemma 5.4]{mueger-galois}.

\begin{proof} The category $\C'$ is symmetric. Let $G$ be a finite group and $u \in Z(G)$, $u^2 = 1$, such that $\C' \cong \Rep (G, u)$.
Since $\C'$ is a symmetric fusion category, for every simple object $Z$ of $\C$ we have $\theta_Z = \pm 1$, and $Z \in \Rep G/(u)$ if and only if $\theta_Z = 1$.

Suppose on the contrary that $\theta_Z = -1$. In particular $\C'$ is not Tannakian and $\E \cong \Rep G/(u)$ is its maximal Tannakian subcategory. Thus $\E$ is a Tannakian subcategory of $\C$ contained in the M\" uger center $\C'$. Let  $\D = \C_{G/(u)}$ denote the de-equivariantization of $\C$, which is an integral fusion category. We shall indicate by $\overline\theta$ the canonical positive balanced structure in $\D$. Thus the group  $G/(u)$ acts on $\D$ by braided autoequivalences preserving the balanced structures, such that $\C \cong \D^{G/(u)}$, and the canonical functor $F:\C \to \D$ is a braided tensor functor which preserves the balanced structures.

\medbreak The essential image of $\C'$ under the functor $F$ is equivalent as a braided fusion category to the category $\svect$ of super-vector spaces; see Remark \ref{ext-sym}. Let $g \in F(\C') \cong \svect$ be the unique non-trivial invertible object; $g$ is the unique simple object of $F(\C')$ such that $\overline\theta_g = -1$.
Note that, since $Z \in \C'$, then $g \in \D'$ \cite[Lemme 2.2 (1)]{bruguieres}.

Since $\theta_Z = -1$ and $F$ preserves the balanced structures, then $\overline\theta_{F(Z)} = -\id_{F(Z)}$. By normality of the functor $F$, this implies that $\Hom_\D(\un, F(Z)) = 0$ and hence $F(Z)$ is isomorphic to a direct sum of copies of the object $g \in F(\C')$.

\medbreak Let $Y$ be a simple constituent of $F(X)$ in $\D$. The object $F(X) \in \D$ decomposes as a direct sum of conjugates of $Y$ under the action of $G/(u)$. The assumption that $Z$ is a constituent of $X \otimes X^*$ implies that $F(Z)$ is a direct summand of $F(X \otimes X^*) \cong F(X) \otimes F(X)^*$.  In particular,  there must exist $s, t \in G/(u)$ such that $g$ has positive multiplicity in the tensor product $\rho^t(Y) \otimes \rho^s(Y)^*$. This implies that
$$g \otimes \rho^s(Y) \cong \rho^t(Y).$$
Hence, $\overline\theta_{g \otimes \rho^s(Y)} = \overline\theta_{\rho^t(Y)} = \overline\theta_Y$. On the other hand, since $g$ centralizes $\rho^s(Y)$, we get from \eqref{bal} that
$$\overline\theta_{g \otimes \rho^s(Y)} = ({\overline\theta_g} \otimes {\overline\theta_{\rho^s(Y)}}) c_{\rho^s(Y), g} \, c_{g, \rho^s(Y)} = - \overline\theta_{\rho^s(Y)} = -\overline\theta_{Y}.$$
This leads to the contradiction $\overline\theta_Y = -\overline\theta_Y$. The contradiction comes from the assumption that $\theta_Z = -1$. Therefore we obtain that $\theta_Z = 1$, as claimed.
\end{proof}

\subsection{Braided fusion categories without non-pointed Tannakian subcategories}
Recall that $\C$ is a weakly integral braided fusion category.
We shall give some sufficient conditions for a tensor product of simple objects of $\C$ to be simple.

\begin{lemma}\label{xy-simple} Suppose that every Tannakian subcategory of $\C$ is pointed.
Let $X, Y$ be simple objects of $\C$ such that $G[X^*] \cap G[Y] = \un$. Assume in addition that $X$ and $Y$ projectively centralize each other.   Then $X \otimes Y$ is a simple object of $\C$.
\end{lemma}

\begin{proof} Suppose $Z$ is a nontrivial common simple constituent of $X^* \otimes X$ and $Y \otimes Y^*$. Since by assumption $G[X^*]\cap G[Y] = \un$, then $Z$ is not invertible. In addition, since $X$ and $Y$ projectively centralize each other, then $Z$ centralizes $X$ and $Y$.

Let $\tilde \C$ be the fusion subcategory of $\C$ generated by $X$ and $Y$. Then $Z$ belongs to the M\" uger center $\tilde\C'$. Lemma \ref{fermionic} implies that $\theta_Z = 1$, hence $Z$ generates a Tannakian subcategory of $\C$. Since $Z$ is not invertible, this contradicts the assumption that $\C$ contains no non-pointed Tannakian subcategories.

Therefore $X^* \otimes X$ and $Y \otimes Y^*$ cannot have non-invertible common simple constituents.  Then the tensor product $X \otimes Y$ is simple, by \cite[Lemma 2.5]{fusion-lowdim}. \end{proof}

\begin{remark}\label{xy-simple-2}  Recall that for every simple object $Z$, the order of the group $G[Z]$ of invertible simple constituents of $Z\otimes Z^*$ divides $(\FPdim Z)^2$ \cite[Lemma 2.2]{fusion-lowdim}.

Suppose that $\C$ is integral. Let $X$ and $Y$ be simple objects of $\C$ such that $\FPdim X$ and $\FPdim Y$ are relatively prime and $X$ and $Y$ projectively centralize each other.  Then the orders of the groups $G[X^*]$ and $G[Y]$ are relatively prime too and Lemma \ref{xy-simple} applies. \end{remark}

\begin{corollary}\label{cor-fpcoprime}
Let $\C$ be an integral braided fusion category such that every Tannakian subcategory of  $\C$ is pointed.
Suppose that $X, Y$ are simple objects of $\C$ such that $\FPdim X$ and $\FPdim Y$ are relatively prime. Then one of the following holds:
\begin{enumerate}
\item[(i)]\label{1} The tensor product $X \otimes Y$ is simple, or
\item[(ii)]\label{2} $S_{X, Y} = 0$.
\end{enumerate} \end{corollary}

\begin{proof} Since $\FPdim X$ and $\FPdim Y$ are relatively prime, \cite[Lemma 7.1]{ENO2} implies that either $S_{X, Y} = 0$, and (ii) holds, or $X$ and $Y$ projectively centralize each other. If the second possibility holds, then $X \otimes Y$ is simple, in view of Lemma \ref{xy-simple} and Remark \ref{xy-simple-2}. Hence (i) holds in this case. \end{proof}

\section{The Frobenius-Perron graphs of a non-degenerate integral fusion category}
Throughout this section $\C$ will be a non-degenerate integral braided fusion category.
As before, we consider $\C$ endowed with its canonical positive spherical structure, so that $\C$ is a modular category.

\begin{example} An example of a non-degenerate integral fusion category is given by the Drinfeld center $\Z(\D)$, where $\D$ is an integral fusion category. In view of \cite[Proposition 4.5]{ENO2}, $\Z(\D)$ is solvable if and only if $\D$ is solvable.

\medbreak Consider the case where $\D$ is itself a non-degenerate braided fusion category, then there is an equivalence of braided fusion categories $\Z(\D) \cong \D \boxtimes \D^{\rev}$. In particular, every simple object of $\Z(\D)$ is isomorphic to one of the form $X \boxtimes Y$, where $X, Y \in  \Irr(\D)$, and we have $\FPdim (X \boxtimes Y) = \FPdim X \, \FPdim Y$. Thus in this case the graph $\Delta(\Z(\D))$ is the complete graph on the vertex set $\pi(\cd(\D))$.   \end{example}

\begin{lemma}\label{s-null} Suppose that every Tannakian subcategory of $\C$ is pointed. Let $X, Y$ be non-invertible simple objects of $\C$. Then we have:
\begin{enumerate}
\item[(i)] If $d(\FPdim X, \FPdim Y) > 2$ in the graph $\Gamma (\C)$, then $S_{X, Y} = 0$.
\item[(ii)]\label{const-xx*}  If $\FPdim X$ and $\FPdim Y$ belong to different connected components of $\Gamma (\C)$, then the multiplicity of $Y$ in $X \otimes X^*$ is determined by the formula 
\begin{equation}\label{nzxx*}\frac{\FPdim \C}{(\FPdim X)^2}\, N^Y_{X, X^*} = \sum_{\FPdim T = 1} S_{Y^*, T}.\end{equation}
\end{enumerate}
\end{lemma}

\begin{proof} (i). The assumption implies that $\FPdim X$ and $\FPdim Y$ are relatively prime. Moreover, the tensor product $X \otimes Y$ cannot be simple, since otherwise we would have $d(\FPdim X, \FPdim Y) \leq 2$ in $\Gamma (\C)$. It follows from Corollary \ref{cor-fpcoprime} that $S_{X, Y} = 0$. This shows (i).

\medbreak (ii). Let $\Gamma_1$ denote the connected component of $\FPdim X$. The multiplicity of $Y$ in $X \otimes X^*$ is given by the Verlinde formula \eqref{verlinde}.

In view of part (i), the nonzero summands in this expression correspond to simple objects $T$ such that $d(\FPdim T, \FPdim X) \leq 2$, so in particular $\FPdim T$ belongs to the connected component $\Gamma_1$ of $\FPdim X$, or else $T$ is invertible (that is, $\FPdim T = 1$). Thus we have
\begin{align*}N_{X, X^*}^Y & = \frac{1}{\FPdim \C}\sum_{d(\FPdim T, \FPdim X) \leq 2} \frac{S_{XT} \, S_{X^*T} \, S_{Y^*T}}{\FPdim T} \\ & + \frac{1}{\FPdim \C} \sum_{\FPdim T = 1} S_{XT} \, S_{X^*T} \, S_{Y^*T}.\end{align*}
Note that if $T$ is an invertible object, then $S_{XT} \, S_{X^*T} = |S_{X, T}|^2 = (\FPdim X)^2$, by \cite[Lemma 6.1]{gel-nik}. On the other hand, if $d(\FPdim T, \FPdim X) \leq 2$, then $\FPdim T$ and $\FPdim Y$ are not connected in $\Gamma(\C)$,  since  by assumption $\FPdim Y$ does not belong to $\Gamma_1$. Hence $S_{Y^*, T} = 0$ in this case, by part (i).

Therefore the above expression reduces to
\begin{equation}N_{X, X^*}^Y = \frac{(\FPdim X)^2}{\FPdim \C} \, \sum_{\FPdim T = 1} S_{Y^*T}.\end{equation}
 This implies (ii) and finishes the proof of the lemma.
 \end{proof}

\begin{remark}\label{rmk-snull} Keep the assumption in Lemma \ref{s-null} (ii). Thus $X$ and $Y$ are non-invertible simple objects of $\C$ such that $\FPdim X$ and $\FPdim Y$ belong to different connected components of $\Gamma (\C)$. Assume in addition that $Y \in \C_{ad}$. Then $S_{Y^*, T} = \FPdim Y$, for every invertible object $T$, because $\C_{ad} = \C_{pt}'$ \cite[Corollary 3.27]{DGNOI}. Hence relation \eqref{nzxx*} becomes 
$$\frac{\FPdim \C}{(\FPdim X)^2}\, N^Y_{X, X^*} = \FPdim Y \, \FPdim \C_{pt}.$$ In particular, we get that $N^Y_{X, X^*} > 0$.
\end{remark}

\begin{proposition}\label{nd-connect} Let $\C$ be a non-degenerate integral braided fusion category and assume that every Tannakian subcategory of  $\C$ is pointed. Then the graph $\Gamma(\C)$ has at most two connected components. \end{proposition}

Observe that the proposition applies, in particular, when the non-degenerate fusion category $\C$ contains no nontrivial Tannakian subcategory.

\begin{proof} Let us first consider the case where for every non-invertible simple object $X$ and for every non-invertible simple constituent $Z$ of $X \otimes X^*$, $\FPdim X$ is connected to $\FPdim Z$ in $\Gamma (\C)$.   

We claim that in this case the graph $\Gamma(\C)$ is connected. Indeed, suppose on the contrary that there exist $X, Y \in \Irr (\C)$ such that $\FPdim X$, $\FPdim Y$ belong to different connected components of $\Gamma (\C)$. In particular $\FPdim X$ and $\FPdim Y$ are relatively prime and therefore $G[X^*]\cap G[Y] = \un$.

Since the tensor product $X \otimes Y$ cannot be simple (otherwise we would have $d(\FPdim X, \FPdim Y) \leq 2$), then it follows from \cite[Lemma 2.5]{fusion-lowdim} that $X^* \otimes X$ and $Y \otimes Y^*$ must have a common simple constituent $Z$, which is necessarily non-invertible. We thus reach a contradiction since, by assumption, the Frobenius-Perron dimensions of non-invertible components of $X^* \otimes X$ (respectively, $Y \otimes Y^*$) belong to the same component as $\FPdim X$ (respectively as $\FPdim Y$). This contradiction shows that $\Gamma(\C)$ must be connected, as claimed. 

\medbreak It remains to consider the case where there exists a non-invertible simple object $X$ and a non-invertible simple constituent $Z$ of $X \otimes X^*$, such that $\FPdim X$ is not connected to $\FPdim Z$ in $\Gamma (\C)$. 

In this case it follows from Lemma \ref{s-null} (ii) that the multiplicity of $Z$ in $X \otimes X^*$ satisfies
$$\frac{\FPdim \C}{(\FPdim X)^2}\, N^Z_{X, X^*} = \sum_{\FPdim T = 1} S_{Z^*, T}.$$
Suppose that $Y$ is a non-invertible simple object such that $\FPdim Y$ does not belong to the connected component of $\FPdim X$ in $\Gamma(\C)$. If $\FPdim Y$ is not connected to $\FPdim Z$, then Lemma \ref{s-null} (ii) also implies that the multiplicity of $Z$ in $Y \otimes Y^*$ satisfies 
$$\frac{\FPdim \C}{(\FPdim Y)^2}\, N^Z_{Y, Y^*} = \sum_{\FPdim T = 1} S_{Z^*, T} = \frac{\FPdim \C}{(\FPdim X)^2}\, N^Z_{X, X^*}.$$
In particular, $N^Z_{Y, Y^*} > 0$ and the following relation holds:
$$(\FPdim X)^2\, N^Z_{Y, Y^*} = (\FPdim Y)^2\, N^Z_{X, X^*}.$$
By assumption, $\FPdim X$ and $\FPdim Y$ are relatively prime, and thus we obtain that $(\FPdim X)^2$ divides $N^Z_{X, X^*}$. This is impossible, because $N^Z_{X, X^*} \FPdim Z < \FPdim (X \otimes X^*) = (\FPdim X)^2$. This contradiction shows that for every non-invertible simple object $Y$, $\FPdim Y$ must be connected either to $\FPdim X$ or to $\FPdim Z$. Therefore $\Gamma(\C)$ has two connected components in this case. This finishes the proof of the proposition.
\end{proof}

The statement of Proposition \ref{nd-connect} can be strengthned as follows in the case where $\C$ has no nontrivial pointed subcategories.

\begin{proposition}\label{npt-connected} Let $\C$ be a non-degenerate integral braided fusion category such that  $\C_{pt} = \vect$. Assume in addition that $\C$ contains no nontrivial Tannakian subcategories. Then the graph $\Gamma(\C)$ is connected. \end{proposition}

Observe that since $\C_{pt} = \vect$, the assumption that $\C$ contains no nontrivial Tannakian subcategories is equivalent to the assumption that every Tannakian subcategory is pointed.

\begin{proof} 
Since $\C$ is non-degenerate,  then $(\FPdim X)^2$ divides $\FPdim \C$ for every simple object $X$ of $\C$ \cite[Theorem 2.11 (i)]{ENO2}.
Suppose on the contrary that $X$ and $Y$ are non-invertible simple objects such that $\FPdim X$ and $\FPdim Y$ belong to different connected components of $\Gamma(\C)$. 
In particular $\FPdim X$ and $\FPdim Y$ are relatively prime and therefore the product $(\FPdim X)^2(\FPdim Y)^2$ divides $\FPdim \C$.

Since $\C$ is non-degenerate, then $\C_{ad} = \C_{pt}'$. Hence the assumption implies that $\C =\C_{ad}$. 
As pointed out in Remark \ref{rmk-snull}, the multiplicity of $Y$ in $X \otimes X^*$ satisfies 
$$\frac{\FPdim \C}{(\FPdim X)^2}\, N^Y_{X, X^*} = \FPdim Y \, \FPdim \C_{pt} = \FPdim Y.$$
Then $$\frac{\FPdim \C}{(\FPdim X)^2(\FPdim Y)^2}\, N^Y_{X, X^*} = \frac{1}{\FPdim Y}.$$
We thus arrive to a contradiction, because the left hand side of this equation is an integer, while the right hand side is not. This shows that the graph $\Gamma(\C)$ must be connected, as claimed. 
\end{proof}

\subsection{Proof of main results on the graph of a braided fusion category}\label{pfs} We now proceed to apply the results in the previous subsections in order to give a proof of Theorems \ref{main-braidedgt} and \ref{main-nondeg-wgt}  on the graphs of braided fusion categories.

\begin{proof}[Proof of Theorem \ref{main-nondeg-wgt}] Let $\C$ be a non-degenerate integral braided fusion category. If every Tannakian subcategory of $\C$ is pointed, then $\Gamma(\C)$ and hence also $\Delta(\C)$, has at most two connected components, by Proposition \ref{nd-connect} (c. f. Remark \ref{graphs-sets}).  We may therefore assume that $\C$ contains a Tannakian subcategory $\E \cong \Rep G$, where $G$ is a non-abelian finite group.  Then the category $\C$ is an equivariantization $\C \cong \D^G$, where $\D = \C_G$ is the associated braided $G$-crossed fusion category (see Subsection \ref{tann-subc}).

It follows from Theorem \ref{thm-leq3} that the graph $\Delta(\C^G)$ has at most three connected components. This proves part (i) of the theorem.
Moreover, if $\C$ is solvable, then the group $G$ is solvable as well. Therefore the graph $\Delta(\C)$ has at most two connected components in view of Theorem \ref{equiv-solv}. This proves part (ii) and finishes the proof of the theorem.
 \end{proof}

\begin{proof}[Proof of Theorem \ref{main-braidedgt}] Let $\C$ be a braided group-theoretical fusion category. In view of a result of Naidu, Nikshych and Witherspoon \cite[Theorem 7.2]{NNW} there exists a Tannakian subcategory $\E \cong \Rep G$ of $\C$ such that the de-equivariantization $\C_G$ is pointed. In particular, $\C$ is an equivariantization of a pointed fusion category. As a consequence of Theorem \ref{eq-pted}, we get part (i) of the theorem, since the prime graph and the common divisor graph of $\C$ have the same number of connected components. Let us show part (ii). Let $\C$ be a group-theoretical non-degenerate braided fusion category and let $\E \cong \Rep G$ be a Tannakian subcategory of $\C$ such that $\C_G$ is pointed. Then $\C_G \cong \C(\Gamma, \omega)$, for some finite group $\Gamma$ and $\omega \in H^3(\Gamma, k^*)$, and therefore $\Z(\C_G) \cong \Z(\C(\Gamma, \omega)) \cong \Rep D^\omega(\Gamma)$.  In addition there is an equivalence of braided fusion categories $\C \boxtimes (\C_G^0)^{\rev} \cong \Z(\C_G)$, and the neutral homogeneous component $\C_G^0 \subseteq \C_G$ is also pointed; see Subsection \ref{tann-subc}. This implies that $\Delta(\C) = \Delta(D^\omega(\Gamma))$ and part (ii) follows from  Theorem \ref{tqd}.
\end{proof}

\section{Application}\label{application}

In this section we shall give a proof of Theorem \ref{appl}. Our proof relies on Proposition \ref{npt-connected}. 
We shall also need the following lemma.

\begin{lemma}\label{sl-deg} Let $\C$ be a braided fusion category. Suppose that
$\C$ contains no nontrivial non-degenerate or Tannakian subcategories. Then $\C$ is
slightly degenerate.
\end{lemma}

\begin{proof} The statement of the lemma is contained in \cite[Lemma 7.1]{witt-wgt}.  \end{proof}

\begin{proof}[Proof of Theorem \ref{appl}] Let $\C$ be a weakly integral braided fusion category satisfying the assumptions of the theorem. That is, the Frobenius-Perron dimension of each simple object of $\C$ is a $p_i$-power, where $p_1, \dots, p_r$ are prime numbers, $r \geq 1$.

The proof follows the lines of the proof of Theorem 7.2 of \cite{witt-wgt}. We argue by induction on $\FPdim \C$. Observe that the assumption on the dimensions of simple objects of $\C$ is also satisfied by any fusion subcategory and, in view of \eqref{dim-equiv}, also by any de-equivariantization of $\C$.

If $\FPdim \C = 1$ there is nothing to prove. If $\C_{ad} \subsetneq \C$, then $\C$ is a $U(\C)$-extension of its adjoint subcategory $\C_{ad}$, where $U(\C)$ denotes the universal grading group of $\C$. Since $\C$ is braided, the group $U(\C)$ is abelian. We may inductively assume that $\C_{ad}$ is weakly group-theoretical (respectively, solvable), and then $\C$ is weakly group-theoretical (respectively, solvable) as well, by \cite[Proposition 4.5]{ENO2}. 

\medbreak We may thus assume that $\C = \C_{ad}$. In particular, $\C$ is integral and not pointed. Furthermore,  we may also assume that the graph $\Delta(\C)$ consists of at least two isolated points; otherwise the Frobenius-Perron dimensions of simple objects of $\C$ are powers of a fixed prime number and the theorem follows from \cite[Theorem 7.2]{witt-wgt}.

\medbreak We shall show below that $\C$ contains a nontrivial Tannakian
subcategory $\E$. Hence  $\E \cong \Rep G$ for some finite group $G$.  Since $\C$ contains $\Rep G$ as a Tannakian subcategory, then $\C$ is an equivariantization of a $G$-crossed braided fusion category $\C_G$.
 Since $\FPdim \C^0_G \leq \FPdim \C_G  = \FPdim \C / |G| < \FPdim \C$ and $\C_G^0$ is a braided fusion category, then it follows by induction that $\C_G^0$ is weakly group-theoretical. Then $\C$ is weakly group-theoretical as well \cite[Proposition 4.1]{witt-wgt}.

Furthermore, if $\C$ satisfies assumptions (a) or (b), then it follows from Formula \eqref{dim-equiv}, that  so does the fusion category $\C_G$ (which is not necessarily braided), and then the same is true for the fusion subcategory $\C_G^0 \subseteq \C_G$. Hence $\C_G^0$ is solvable by induction. On the other hand, assumption (a) implies that the prime graph $\Delta(G)$ consists of at most two isolated points, while assumption (b) implies that the connected components of the graph $\Delta(G)$ are among $\{p_1\}, \dots, \{p_r\}$, where $p_i \neq 2, 3, 5$ or $7$, for all $i = 0, \dots, r$. In any case, we get that the group $G$ is solvable, by \cite[Corollary 1]{manz-sw}. Therefore $\C$ is solvable, by \cite[Proposition 4.1]{witt-wgt}.

\medbreak Suppose on the contrary that $\C$ contains no nontrivial Tannakian subcategory. 
We may assume that $\C$ contains no nontrivial proper non-degenerate subcategories. In fact, if $\vect \subsetneq \D \subsetneq \C$ is  non-degenerate, then  $\C \cong \D \boxtimes \D'$.  Since $\FPdim \D, \FPdim \D' <
\FPdim \C$, it follows by induction that $\D$ and $\D'$ are both weakly group-theoretical and therefore so is $\C$.
Similarly, if $\C$ satisfies (a) or (b), then so do $\D$ and $\D'$. Hence $\D$ and $\D'$ are both solvable, by induction,  and thus $\C$ is solvable too.

\medbreak Suppose next that $\C$ is itself non-degenerate. Since $\C = \C_{ad}$, then $\C_{pt} = \vect$ and Proposition \ref{npt-connected} implies that the graph $\Delta(\C)$ is connected, which is a contradiction.

\medbreak It remains to consider the case where $\C$ contains no nontrivial Tannakian or non-degenerate subcategories. In this case, Lemma \ref{sl-deg} implies that $\C$ is slightly degenerate. Since $\Delta(\C)$ consists of at least two isolated points, then $\C$ has a  simple object of $p$-power dimension for some odd prime number $p$. This contradicts the assumption that  $\C$ has no nontrivial Tannakian subcategories, in view of \cite[Proposition 7.4]{ENO2}. The proof of the theorem is now complete.  \end{proof}


\begin{thebibliography}{9999}

\bibitem{alfandary} G. Alfandary, \emph{On graphs related to conjugacy classes of groups},  Isr. J. Math. \textbf{86},  211--220 (1994).

\bibitem{BK} B. Bakalov, A. Kirillov Jr., \emph{Lectures on tensor categories and modular functors}, University Lecture Series \textbf{21}, Amer. Math. Soc., 2001.

\bibitem{BHM} E. Bertram, M. Herzog, A. Mann, \emph{On a graph related to conjugacy classes of groups},
Bull. Lond. Math. Soc. \textbf{22}, 569--575 (1990).

\bibitem{brown} K. Brown \emph{Cohomology of Groups}, Graduate Texts in Mathematics \textbf{87}, Springer-Verlag, 1982.

\bibitem{bruguieres} A. Brugui\` eres, \emph{Cat\' egories pr\' emodulaires, modularisations et invariants des vari\' et\' es de
dimension 3}, Math. Ann. \textbf{316},  215--36 (2000).

\bibitem{br-bu} A. Brugui\` eres, S. Burciu, \emph{On normal tensor functors and coset decompositions for fusion categories}, preprint  \texttt{arXiv:1210.3922} (2013).

\bibitem{tensor-exact} A. Brugui\` eres, S. Natale, \emph{Exact sequences of tensor categories},
Int. Math. Res. Not. \textbf{2011}, 5644--5705 (2011).

\bibitem{fusionrules-equiv} S. Burciu, S. Natale, \emph{Fusion rules of equivariantizations of fusion categories}, J. Math. Phys.  \textbf{54}, 013511 (2013).

\bibitem{int-mod} P. Bruillard, C. Galindo, S.-M. Hong, Y. Kashina, D. Naidu, S. Natale, J. Plavnik, E. Rowell,  \emph{Classification of integral modular categories of Frobenius-Perron dimension $pq^4$ and $p^2q^2$}, Can. Math. Bull.  \textbf{57}, 721--734 (2014).

\bibitem{camina} A. R. Camina, R. D. Camina, \emph{The influence of conjugacy class sizes on the  structure of finite groups: a survey}, Asian-Eur. J. Math. \textbf{4}, 559--588 (2011).

\bibitem{casolo-dolfi} C. Casolo, S. Dolfi, \emph{The diameter of a conjugacy class graph of finite groups}, Bull. Lond. Math. Soc. \textbf{28}, 141--148 (1996).

\bibitem{DMNO} A. Davydov, M. Mueger, D. Nikshych,
V. Ostrik, \emph{The Witt group of non-degenerate braided fusion categories},
J. Reine Angew. Math. \textbf{677}, 135--177 (2013).

\bibitem{deligne} P. Deligne, \emph{Cat\' egories tensorielles}, Mosc. Math. J. \textbf{2}, 227--248 (2002).

\bibitem{dpr}  R. Dijkgraaf, V. Pasquier, Ph. Roche,  \emph{Quasi-Quantum Groups Related to Orbifold Models}, Proceedings of the ``International Colloquium on Modern Quantum Field Theory'', Tata Institute of Fundamental Research, 375--383 (1990).


\bibitem{fusion-lowdim}  J. Dong, S. Natale, L. Vendramin, \emph{Frobenius property for fusion categories of small integral dimension},  J. Algebra Appl. \textbf{14} 1550011  (2015).

\bibitem{DGNO}  V. Drinfeld, S. Gelaki, D. Nikshych, V. Ostrik,
\emph{Group-theoretical properties of nilpotent modular categories}, preprint
\texttt{arXiv:0704.0195} (2007).

\bibitem{DGNOI}  V. Drinfeld, S. Gelaki, D. Nikshych, V. Ostrik,
\emph{On braided fusion categories I}, Sel. Math. New Ser.  \textbf{16}, 1--119 (2010).

\bibitem{ENO}  P. Etingof, D. Nikshych, V. Ostrik,
\emph{On fusion categories},
Ann. Math \textbf{162}, 581--642 (2005).


\bibitem{ENO2}  P. Etingof, D. Nikshych, V. Ostrik,
\emph{Weakly group-theoretical and solvable fusion categories},
Adv. Math \textbf{226},    176--205   (2011).

\bibitem{GNN} S. Gelaki, D. Naidu, D. Nikshych \emph{Centers of graded fusion categories}, Algebra Number Theory
\textbf{3}, 959--990 (2009).

\bibitem{gel-nik} S. Gelaki, D. Nikshych, \emph{Nilpotent fusion categories}, Adv. Math. \textbf{217},  1053--1071 (2008).


\bibitem{isaacs} M. Isaacs, \emph{Character theory of finite groups}, Academic Press, New York, 1976.

\bibitem{isaacs2} M. Isaacs, \emph{Coprime group actions fixing all nonlinear irreducible characters},
Can. J. Math. \textbf{41}, 68--82 (1989).

\bibitem{ipgraph} M. Isaacs, C. Praeger, \emph{Permutation group subdegrees and the common divisor graph}, J. Algebra \textbf{159}, 158--175 (1993).

\bibitem{ito} N. Ito, \emph{On the degrees of irreducible representations of a finite group}, Nagoya Math. J. \textbf{3}, 5--6 (1951).

\bibitem{kazarin} L. Kazarin, \emph{On groups with isolated conjugacy classes},  Sov. Math. \textbf{25},  43--49 (1981); translation from Izv. Vyssh. Uchebn. Zaved., Mat. 1981, (230) 40--45 (1981).

\bibitem{lewis} M. Lewis, \emph{An overview of graphs associated with character degrees and conjugacy class sizes in finite groups},  Rocky Mountain J. Math. \textbf{38}, 175--211 (2008).

\bibitem{manz} O. Manz, \emph{Degree Problems II: $\pi$-separable character degrees}, Commun. Algebra \textbf{13}, 2421--2431 (1985).

\bibitem{manz-sw} O. Manz, R. Staszewski, W. Willems, \emph{The number of components of a graph related to character degrees}, Proc. Amer. Math. Soc. \textbf{103}, 31--37 (1988).

\bibitem{manz-ww} O. Manz, W. Willems, T. Wolf, \emph{The diameter of the character degree graph}, J. Reine Angew. Math. \textbf{402}, 181--198 (1989).

\bibitem{manz-wolf} O. Manz, T. Wolf, \emph{Representations of solvable groups}, London Math. Soc. Lect. Not. Ser. \textbf{185}, Cambridge University Press, Cambridge, 1993.

\bibitem{mason-ng} G. Mason, S. Ng, \emph{Group cohomology and gauge equivalence of some twisted quantum doubles}, Trans. Amer. Math. Soc, \textbf{353}, 3465--3509 (2001).

\bibitem{michler} G. Michler, \emph{A finite simple group of Lie type has $p$-blocks with different defects, $p \neq 2$},  J. Algebra \textbf{104}, 220--230 (1986).

\bibitem{mueger-galois} M. M\" uger, \emph{Galois theory for braided tensor categories and the modular closure},
Adv. Math. \textbf{150}, 151--201 (2000).

\bibitem{mueger-gcrossed} M. M\" uger, \emph{Galois extensions of braided tensor categories and braided crossed
$G$-categories}, J. Algebra \textbf{277}, 256--281 (2004).

\bibitem{NNW} D. Naidu, D Nikshych, S. Witherspoon \emph{Fusion subcategories of representation categories of twisted quantum doubles of finite groups}, Int. Math. Res. Not. \textbf{22}, 4183--4219 (2009).

\bibitem{NR} D. Naidu, E. Rowell, \emph{A finiteness property for braided fusion categories},  Algebr. Represent.Theory \textbf{14}, 837--855 (2011).

\bibitem{ext-ty} S. Natale, \emph{Hopf algebra extensions of group algebras and Tambara-Yamagami categories}, Algebr. Represent. Theory \textbf{13},  673--691 (2010).

\bibitem{witt-wgt} S. Natale, \emph{On weakly group-theoretical non-degenerate braided fusion categories}, to appear in J. Noncommut. Geom. Preprint  	\texttt{arXiv:1301.6078}.

\bibitem{faithful-fusion} S. Natale, \emph{Faithful simple objects, orders and gradings of fusion categories}, Algebr. Geom. Topol. \textbf{13},  1489--1511 (2013).

\bibitem{tambara}  D. Tambara,
\emph{Invariants and semi-direct products for finite group actions
on tensor categories}, J. Math. Soc. Japan \textbf{53},
429--456 (2001).

\bibitem{turaev-b} V. Turaev, \emph{Quantum invariants of knots and 3-manifolds}, de Gruyter Studies in Math. \textbf{18}, Berlin, 1994.

\bibitem{turaev} V. Turaev, \emph{Homotopy quantum field theory}, EMS Tracts in Mathematics \textbf{10}, European Mathematical Society, Zurich, 2010.

\bibitem{turaev2} V. Turaev, \emph{Crossed group-categories}, Arabian Journal
for Science and Engineering \textbf{33}, 484--503 (2008).


\bibitem{willems} W. Willems, \emph{Blocks of defect zero and degree problems}, Proc. Sympos. Pure Math. \textbf{47}
(I), 481--484,  Amer. Math. Soc, Providence, R. I., 1987.

\bibitem{yuster} T. Yuster, \emph{Orbit sizes under automorphism actions in finite groups},
J. Algebra \textbf{82}, 342--352 (1983).

\end{thebibliography}
\end{document}